\documentclass[12pt]{amsart}

\usepackage{color}
\usepackage{amsbsy}
\usepackage{amssymb,amsmath,amstext,mathrsfs}
\usepackage{graphicx}

\makeatletter
\def\verbatim{\interlinepenalty\@M \@verbatim
  \leftskip\@totalleftmargin\advance\leftskip2pc
  \frenchspacing\@vobeyspaces \@xverbatim}
\makeatother
\newtheorem{thm}{Theorem}[section]
\newtheorem{cor}[thm]{Corollary}
\newtheorem{lem}[thm]{Lemma}
\newtheorem{prop}[thm]{Proposition}
\newtheorem{claim}[thm]{Claim}

\theoremstyle{definition}
\newtheorem{defn}{Definition}[section]

\theoremstyle{remark}
\newtheorem{rem}{Remark}[section]

\numberwithin{equation}{section}

\newtheorem{ex}[thm]{Example}

\DeclareMathOperator{\sgn}{sgn}

\newcommand{\ca}{\mathrm{card} \hspace{.02in}}

\newcommand{\hu}{\hspace{.02in}}

\newcommand{\va}{\left|}
\newcommand{\vb}{\right|}
\newcommand{\vd}{\right\| }
\newcommand{\vc}{\left\| }

\newcommand{\pl}{\left(}
\newcommand{\pr}{\right)}
\newcommand{\ql}{\left[}
\newcommand{\qr}{\right]}
\newcommand{\tl}{\left\{}
\newcommand{\tr}{\right\}}

\newcommand{\dg}{\epsilon-\mathbf{DP} \pl X, Y \pr }
\newcommand{\wc}{\mathbf{WCM} \pl X,Y \pr }

\newcommand{\dk}{\epsilon-\mathbf{DP}  \pl X, \mathbb{K} \pr }
\newcommand{\wk}{\mathbf{WCM} \pl X, \mathbb{K} \pr }

\newcommand{\cl}{\mathrm{cl} \hu}

\newcommand{\ra}{\longrightarrow}

\title[Stability and instability]{Stability and instability of weighted composition operators}

\author{Jes\'us Araujo}
\address{Departamento de Matem\'aticas,
Estad\'{\i}stica y Computaci\'on\\ Universidad de Cantabria\\
Facultad de Ciencias\\ Avda.
de los Castros, s. n.\\ E-39071 Santander, Spain}

\email{araujoj@unican.es}
\thanks{Research of the first author was partially supported by
the Spanish Ministry of Science and Education (Grants numbers MTM2004-02348 and MTM2006-14786).}

\author{Juan J. Font}
\address{Departamento de Matem\'aticas\\
Universitat Jaume I\\ Campus Riu Sec\\ 8029 AP, Castell\'on, Spain}
\email{font@mat.uji.es}
\thanks{Research of the second author was partially supported by
European Union (FEDER) and the
Spanish Ministry of Science and Education (Grant number
MTM2004-07665-C02-01), and by Bancaixa (Projecte P1-1B2005-22
codi 05I343).}

\date{}
\begin{document}

\begin{abstract}
Let $\epsilon >0$.  A continuous linear operator
 $T:C(X) \ra C(Y)$ is said to be {\em
$\epsilon$-disjointness preserving} if $\vc (Tf)(Tg)\vd_{\infty}
\le \epsilon$, whenever  $f,g\in C(X)$ satisfy $\vc f\vd_{\infty}
=\vc g\vd_{\infty} =1$ and $fg\equiv 0$. In this paper we address
basically two main questions:

1.- How close there must be a weighted composition operator  to a given
$\epsilon$-disjointness preserving operator?

2.- How far can the set of weighted composition operators be from a
given $\epsilon$-disjointness preserving operator?

We address these two questions distinguishing among three cases:
$X$ infinite, $X$ finite, and $Y$ a singleton
($\epsilon$-disjointness preserving functionals).

We provide sharp stability and instability bounds for the three
cases.
\end{abstract}

\maketitle

\section{Introduction}

Suppose that a mathematical object satisfies a certain property
approximately. Is it then possible to approximate this object by
objects that satisfy the property exactly? This stability problem
appears in almost all branches of mathematical analysis and is of
particular interest in probability theory and in the realm of
functional equations. Within this context, considerable attention
has been mainly given to approximately multiplicative maps (see
\cite{Jo1}, \cite{Jo}, \cite{Jar}, and \cite{Se}) and to
approximate isometries (see \cite{HU}, \cite{HU2}, \cite{Bour},  and
\cite{HIR}).

\medskip
Recently, G. Dolinar (\cite{Dol}) treated a more general problem
of stability concerning a  kind of operators which "almost"
preserves  the disjointness of cozero sets (see Definition~\ref{djp}).

\bigskip
We need some notation. Let $\mathbb{K}$ denote the field of real or
complex numbers. Topological spaces $X$ and $Y$ are assumed to be compact
and Hausdorff. Also $C(X)$
stands for the Banach space of all $\mathbb{K}$-valued continuous
functions defined on $X$, equipped with its usual supremum norm.

\begin{defn}
An operator $S: C(X) \ra C(Y)$ is said to be a {\em weighted composition map}  if there exist $a \in C(Y)$ and a  map $h: Y \ra X$, continuous on
$c(a) := \tl y \in Y : a(y) \neq 0 \tr  $,
 such that  \[(Sf) (y) = a(y) f(h(y))\]
  for every $f \in C(X)$
and $y \in Y $.
\end{defn}

Obviously every weighted composition map is linear and continuous. We also include the case that $S \equiv 0$
as a weighted composition map (being $c(a) = \emptyset$).

\medskip

Recall that a linear operator $T:C(X)\longrightarrow C(Y)$ is said to be {\em disjointness preserving} (or {\em
separating})  if,
given $f,g\in C(X)$, $fg\equiv 0$ yields $(Tf)(Tg)\equiv 0$. Clearly every weighted composition map is disjointness
preserving. Reciprocally, it is well known that if a  disjointness preserving operator is {\em continuous}, then it is a weighted composition. On the other hand, automatic continuity of disjointness preserving operators can be obtained sometimes (see for instance \cite{Jz}, \cite{ABN}, \cite{FH}, \cite{JW})).

\begin{defn}\label{djp}
Let $\epsilon >0$. A continuous linear operator
 $T:C(X)\longrightarrow C(Y)$ is said to be {\em
$\epsilon$-disjointness preserving} if $\vc (Tf)(Tg)\vd_{\infty}  \le \epsilon$, whenever  $f,g\in C(X)$ satisfy
$\vc f\vd_{\infty} =\vc g\vd_{\infty} =1$ and $fg\equiv 0$ (or, equivalently, if $\vc (Tf)(Tg)\vd_{\infty}  \le \epsilon \vc f \vd_{\infty} \vc g \vd_{\infty}$ whenever $fg\equiv 0$).
\end{defn}

Obviously the study of
$\epsilon$-disjointness preserving operators can be restricted to
those of norm $1$, because if $T \neq 0$ is $\epsilon$-disjointness
preserving, then $T/\vc T \vd$ is $\epsilon/\vc T
\vd^2$-disjointness preserving. On the other hand, every such $T$ has the trivial weighted composition map $S \equiv 0$ at distance
$1$. That is,  giving any bound equal to or bigger than $1$ does not provide any information on the problem.
Apart from this,
 it can be easily checked that
every continuous linear functional on $C(X)$ of norm $1$ is
$1/4$-disjointness preserving and, consequently, every continuous
linear map $T:C(X)\longrightarrow C(Y)$ with $\vc T\vd  =1$ is
$1/4$-disjointness preserving. Thus, if we consider again the trivial weighted composition map
$S\equiv 0$, then
$\vc T-S\vd  =1$. We conclude that our study can be restricted to $\epsilon$ belonging to the interval $(0, 1/4)$.

\medskip

In  \cite{Dol} the author, following the above stability questions,
studies when an $\epsilon$-disjointness preserving operator is
close to
a weighted composition map. The main result in \cite{Dol} reads as follows:
Let $\epsilon > 0$ and let $T:C(X)\longrightarrow C(Y)$ be an
 $\epsilon$-disjointness preserving operator with
$\vc T\vd  =1$. Then there exists a
 weighted composition map $S:C(X)\longrightarrow C(Y)$ such that
\[\vc T-S\vd  \le 20\sqrt{\epsilon}.\]

In view of the above comments we conclude that Dolinar's result is meaningful only for $\epsilon \in \pl 0, 1/400 \pr$.

Apart from the general case, Dolinar also concentrates on the study of linear and continuous functionals,  where the bound given is  $3 \sqrt{\epsilon}$ (see \cite[Theorem 1]{Dol}).

\medskip

On the other hand,
notice that  when  $X$ has just one point, we are in a situation of "extreme stability", because every
continuous linear operator
is a weighted composition map. But in general,
given an $\epsilon$-disjointness preserving operator, we do not necessarily have a weighted composition map arbitrarily close. Instability questions  deal with bounds of how far apart
an $\epsilon$-disjointness preserving operator can be from all weighted composition maps.

\bigskip

In the present paper we improve Dolinar's result by showing, under
necessary restrictions on $\epsilon$, that a weighted composition map is
indeed much closer. If fact we address the following two questions.
Given any $\epsilon$-disjointness preserving operator,
\begin{enumerate}
\item\label{domench} {\sc Stability.} {\em How  close} there must be a weighted composition map? That is,
find the
shortest distance at which we can be certain that there exists a  weighted composition map.
\item\label{lironcaretu} {\sc Instability.} {\em How far} the set of all weighted composition maps can be?
That is,
find the
longest distance at which we cannot be certain that there exists a  weighted composition map.
\end{enumerate}

\medskip

{\sc How close.}
We prove that, for every $\epsilon < 2/17$,  the number $ \sqrt{17\epsilon /2}$ is a bound valid for every $X$ and $Y$ (Theorem~\ref{rz}). It is indeed the smallest in every case, as  we give an example such that, for every $\epsilon <2/17$,   no number strictly less than  $\sqrt{17 \epsilon/2}$  satisfies it (Example \ref{gustavo}).

The question appears to be very related to the following: Find the biggest set $\mathbb{I} \subset (0, 1/4)$ such that every
$\epsilon \in \mathbb{I}$ has the following property: {\em Given an
 $\epsilon$-disjointness preserving operator $T: C(X)
\longrightarrow C(Y)$ with $\vc T \vd =1$, there exists a weighted
composition map  $S: C(X) \longrightarrow C(Y)$ such that $\vc T - S \vd
<1$.} We prove that  $\mathbb{I}= (0, 2/17)$ (Theorem~\ref{rz} and Example~\ref{andalu}).

We will also study the particular case when $X$ is finite. Here the bound, which can be given for every $\epsilon <1/4$
and  every $Y$,
is the number $2 \sqrt{\epsilon}$, and is sharp (Theorem~\ref{opodel} and Example~\ref{llegomalenayamigos2}).

\medskip

{\sc How far.}
Of course,  an answer valid for every case would be trivial, because if we take $X$ with just one point, then every
continuous linear operator is a weighted composition map, so the best bound is just $0$. If  we avoid this trivial case and require
$X$ to have at least two points, then we can see that again the problem turns out to be trivial since the best bound is now attained for sets with two points. The same happens if we
require the set $X$ to have at least $k$ points.

In general,  it can be seen that the answer does not depend on the topological features
of the spaces but on their cardinalities. If we assume that $Y$ has at least two points, then
the number
 $2 \sqrt{ \epsilon}$ is a valid bound if $X$ is infinite  (Theorem~\ref{ex3}), and a
 different value plays the same r\^{o}le  for each finite set $X$ (Theorem~\ref{recero}).

We also prove that these estimates are sharp in every case (Theorems~\ref{vr} and \ref{cero}). But here,
instead of providing a concrete counterexample, we can show that the bounds are best for a general family
of spaces $Y$, namely, whenever $Y$ consists of the
Stone-\v{C}ech compactification of any discrete space.

On the other hand, unlike the previous question, the answer can be given for every $\epsilon < 1/4$.

\medskip

{\sc The case of continuous linear functionals.}
The context when $Y$ has just one
point, that is, the case of continuous linear functionals, deserves to be studied separately. We do this in  Sections~\ref{bound} and \ref{ninguno}. In fact some results given in this case will be tools for a more general study. Various  situations appear in this context, depending on  $\epsilon$. Namely, if
$\epsilon<1/4$,  then the
results depend on an suitable splitting of the interval $(
0, 1/4 )$ (based on the sequence $(\omega_n)$ defined below), as well as on the cardinality of $X$ (Theorem~\ref{martakno}).

Also, as we mentioned above,  when  $X$ has just one point,  every  element of $C(X)'$
is a weighted composition map, that is,  a scalar multiple of the evaluation functional $\delta_x$. We will see that
a related phenomenon sometimes arises  when $X$ is finite (see Remark~\ref{nadenas}).

In every case our results are sharp.

\medskip

{\bf Notation.}
Throughout $\mathbb{K} = \mathbb{R}$ or $\mathbb{C}$. $X$ and $Y$ will be (nonempty) compact Hausdorff spaces.
To avoid the trivial case, we will always assume that $X$ has at least two points.
In a Banach space $E$, for $e \in E$ and $r>0$, $B (e, r)$ and $\overline{B} (e, r)$ denote the open and the closed
balls of center $e$ and radius $r$, respectively.

\smallskip

{\em Spaces and functions.} Given any compact Hausdorff space $Z$, we denote by $\ca Z$ its cardinal. $C(Z)$ will be the Banach space of all $\mathbb{K}$-valued continuous functions on $Z$, endowed with the sup norm $\vc \cdot \vd_{\infty}$. $C(Z)'$ will denote the space of linear and continuous functionals defined on $C(Z)$. If $a \in \mathbb{K}$, we denote by $\widehat{a}$ the constant function equal to $a$ on $Z$. In the special case of the constant function equal to $1$, we denote it by ${\bf 1}$. For $f \in C(Z)$, $0 \le f \le 1$ means that $f(x) \in [0,1]$ for every $x \in Z$.  Given
$f\in C(Z)$, we will consider that
$c(f) =\tl x \in Z : f(x) \neq 0 \tr$ is its cozero set, and ${\rm supp}(f)$ its support. Finally, if $A \subset Z$, we denote by $\cl A$ the closure of $A$ in $Z$, and by $\xi_A$ the characteristic function of $A$.

\smallskip

{\em Continuous linear functionals and measures: $\lambda_{\varphi}$, $\va \lambda \vb$, $\delta_x$.}
For  $\varphi \in C(X)'$, we will write
$\lambda_{\varphi}$   to denote the measure which represents it. For a regular measure
$\lambda$, we will denote by $\va \lambda \vb$ its total variation. Finally, for $x \in X$, $\delta_x$ will be the evaluation functional at $x$, that is, $\delta_x (f) := f(x)$ for every $f \in C(X)$.

\smallskip

{\em The linear functionals $T_y$ and the sets $Y_r$.}
Suppose that $T: C(X) \ra C(Y)$ is linear and continuous.
Then, for each $y\in Y$, we define a continuous linear functional $T_y$ as
$T_{y}(f):=(Tf)(y)$ for every $f \in C(X)$. Also, for each $r \in \mathbb{R}$ we define $Y_r :=\left\{y\in Y:\vc T_{y}\vd  > r \right\}$, which is an open set.
It is clear that, if $\vc T\vd  =1$, then $Y_r$ is nonempty for each $r<1$.

\smallskip

{\em The sets of operators.}
We denote by $\dg$
the set of all $\epsilon$-disjointness preserving operators from $C(X)$ to $C(Y)$,
and by $\wc$
the set of all weighted composition maps from $C(X)$ to $C(Y)$. When $Y$ has just one point, then  $\dg$ and $\wc$ may be viewed as subspaces of $C(X)'$. In this case, we will use the notation $\dk$ and $\wk$ instead of $\dg$ and $\wc$, respectively. That is,
$\dk $ is the space of all $\varphi \in C(X)'$ which satisfy $\va \varphi (f) \vb \va \varphi (g) \vb \le \epsilon$ whenever $f , g \in C(X)$ satisfy $\vc f \vd_{\infty} = 1 = \vc g \vd_{\infty}$ and $f g \equiv 0$, and $\wk$ is the subset of $C(X)'$ of elements of the form $\alpha \delta_x$, where $\alpha \in \mathbb{K}$ and $x \in X$.

\smallskip

{\em The sequences $(\omega_n)$ and $\pl \mathbb{A}_n \pr $.} We define, for each $n \in \mathbb{N}$, $$\omega_n := \frac{n^2 -1}{4 n^2}$$
and
\[\mathbb{A}_n: = \ql \omega_{2n-1}, \omega_{2n+1} \pr , \]
It is clear that $\pl \mathbb{A}_n \pr$ forms a partition of the interval $[0, 1/4)$.

The sequences $(\omega_n)$ and $\pl \mathbb{A}_n \pr $ will determine bounds in Sections~\ref{leffe} and \ref{puntodencuentro}.

\section{Main results I: How close. The general case}

In this section we
give the best stability bound in the general case.
This result is valid for every $X$ in general, assuming no restrictions on cardinality (see Section~\ref{paulus-15nov07} for the proof).

\begin{thm}\label{rz}
Let  $0 < \epsilon
< 2 /17$, and let $T \in \dg$  with $\vc  T\vd =1$.
Then
$$\overline{B} \pl T, \sqrt{\frac{17 \epsilon}{2}} \pr \cap \wc \neq \emptyset.$$
\end{thm}

\begin{rem}\label{ensinkafetefever}
Theorem~\ref{rz} is accurate in two ways.
On the one hand, for every $\epsilon \in (0, 2/17)$, the above bound is sharp, as it can be seen in Example~\ref{gustavo}.
 On the other hand, we have that $(0, 2/17)$ is the maximal interval we can get a meaningful answer in. Namely, if
$\epsilon \ge 2/17$, then  it may be the case that $\vc T -S \vd \ge 1$ for every weighted composition map $S$
(Example~\ref{andalu}). But,
as it is explained in the comments after Definition~\ref{djp},
 this is not a proper answer for the stability question.
\end{rem}

\section{Main results II: How far. The case when $X$ is infinite}\label{004}

We  study  instability first when $X$ is infinite. Our results  depend on whether or not the space $X$ admits an  appropriate measure. 

\begin{thm}\label{ex3}
Let $0 < \epsilon < 1/4$. Suppose that $Y$ has at least two points,  and  that $X$  is infinite. Then for each $t<1$,  there exists
$T \in \dg$ with $\vc T\vd =1$ such that
$$B \pl T,  2 t \sqrt{\epsilon} \pr \cap \wc = \emptyset.$$

Furthermore, if $X$ admits an atomless regular Borel  probability  measure, then $T$ can be taken such that $$B \pl T,  2\sqrt{\epsilon} \pr \cap \wc = \emptyset.$$
\end{thm}

We also see that the above bounds are  sharp when considering some families of spaces $Y$.

\begin{thm}\label{vr}
Let  $0 < \epsilon
< 1/4$. Suppose that $Y$ is the  Stone-\v{C}ech compactification of a discrete space with at least two points, and that $X$ is infinite. Let $T \in \dg$ with $\vc  T\vd =1$.
Then
$$\overline{B} \pl T,  2\sqrt{\epsilon} \pr \cap \wc \neq \emptyset.$$

Furthermore, if $X$ does not admit an atomless regular Borel  probability  measure and $Y$ is finite (with $\ca Y \ge 2$), then
$$B \pl T,  2\sqrt{\epsilon} \pr \cap \wc \neq \emptyset.$$
\end{thm}

The proofs of both results are given in Section~\ref{reallyfar}.

\begin{rem}
The property of admitting an atomless regular Borel probability (or, equivalently, complex and nontrivial) measure can be characterized in purely topological terms. A compact Hausdorff space admits such a  measure if and only if is {\em scattered} (see \cite[Theorem 19.7.6]{S}).
\end{rem}

\section{Main results III: How far and how close. The case when $X$ is finite}\label{leffe}

Next we study the case when $X$ is finite. Here, the best instability bounds depend on the sequence $(\omega_n)$,
and the cardinality of $Y$ does not play any r\^{o}le as
long as it is at least $2$.
We  define  $o'_X :  (0, 1/4) \ra \mathbb{R}$, for every finite set $X$ (recall that we are assuming $\ca X \ge 2$). We put
\begin{displaymath}
o'_X (\epsilon)  := \left\{ \begin{array}{rl} 2 \sqrt{\frac{(n-1)\epsilon}{n+1}} & \mbox{if }  n := \ca X \mbox{ is odd and } \epsilon \le \omega_n
\\
\frac{n-1}{n} & \mbox{if }  n := \ca X \mbox{ is odd and } \epsilon > \omega_n
\\
\frac{2(n-1)  \sqrt{\epsilon}}{n}  & \mbox{if }  n := \ca X \mbox{ is even}
\end{array} \right.
\end{displaymath}

\begin{thm}\label{recero}
Let $0 < \epsilon < 1/4$. Assume that $Y$ has at least two points, and that $X$ is finite.
Then there exists
$T \in \dg$ with $\vc T\vd =1$ such that
$$B \pl T,  o'_X (\epsilon ) \pr \cap \wc = \emptyset.$$
\end{thm}

The next result says that Theorem~\ref{recero} provides a sharp bound, and gives a whole family of spaces $Y$
for which the same one is  a bound for stability as well.  As we can see in Example~\ref{llegomalenayamigos2}, our  requirement on these $Y$ is not superfluous.

\begin{thm}\label{cero}
Let $0 < \epsilon < 1/4$. Suppose that $Y$ is the  Stone-\v{C}ech compactification of a discrete space with at least two points, and that $X$  is  finite.
Let $T \in \dg$ with $\vc  T\vd =1$.
Then
$$\overline{B} \pl T,  o'_X ( \epsilon ) \pr \cap \wc \neq \emptyset.$$
\end{thm}

The instability bounds  are special when the space $X$ is finite. In the following theorem we study  the stability bounds in this particular case. Example~\ref{llegomalenayamigos2} shows that the result  is sharp.

\begin{thm}\label{opodel}
Let $0 < \epsilon < 1/4$. Suppose that
that $X$  is  finite, and let $T \in \dg$  with $\vc  T\vd =1$.
Then
$$\overline{B} \pl T, 2 \sqrt{\epsilon} \pr \cap \wc \neq \emptyset.$$
\end{thm}

Theorems~\ref{recero} and \ref{cero} are proved in Section~\ref{cerilla}, and Theorem~\ref{opodel} in Section~\ref{myrna}.

\section{Main results IV: The case of continuous linear functionals}\label{puntodencuentro}

In some of the previous results, we assume that the space $Y$ has at least two points.
Of course the case when $Y$ has just one point can be viewed as the study of  continuous linear functionals.
In this section we give the  best stability and instability bounds in this case, and see that both bounds coincide.  Here we do not require $X$ to be finite, and Theorem~\ref{martakno} is valid both for $X$ finite and infinite.  Anyway, the result depends on the sequence $(\omega_n)$ and its relation to the cardinal of $X$.

We first introduce the map $o_X : (0, 1/4) \ra \mathbb{R}$ as follows: For $n \in \mathbb{N}$ and  $\epsilon \in \mathbb{A}_n $,
$$o_X (\epsilon)  := \left\{ \begin{array}{rl} \frac{2n-1 -  \sqrt{1 - 4 \epsilon}}{2n} & \mbox{if }   2 n \le \ca X
\\
\frac{k-1 - \sqrt{1 -4 \epsilon}}{k} & \mbox{if } k := \ca X < 2 n \mbox{ and } k \mbox{ is even}
\\
 \frac{k-1}{k} & \mbox{if } k := \ca X < 2 n \mbox{ and } k \mbox{ is odd}
\end{array} \right. $$

We  use this map to give a bound both for stability and instability (see Section~\ref{ninguno} for the proof).

\begin{thm}\label{martakno}
Let $0 < \epsilon < 1/4$.
If $\varphi \in \dk$ and $\vc  \varphi \vd =1$, then $$\overline{B} \pl \varphi, o_X (\epsilon) \pr \cap \wk \neq \emptyset.$$

On the other hand, there exists $\varphi \in \dk$ with $\vc  \varphi \vd =1$ such that
$$B \pl \varphi,  o_X (\epsilon) \pr \cap \wk = \emptyset.$$
\end{thm}

\begin{rem}\label{nadenas}
Sometimes   the information given by  the number $\epsilon$ is redundant,  in that
$\epsilon $ is too "big" with respect to the cardinal of $X$.
 This  happens for instance when $X$ is a set of  $k$ points,
where $k \in \mathbb{N}$ is odd.
This is the reason why the definition of  $o_X$ (and that of $o'_X$) does not necessarily depend on $\epsilon$.
\end{rem}

\section{The bounds $1/4$ and  $2/9$ for continuous linear functionals}\label{bound}

We start with a lemma that will be broadly used.

\begin{lem} \label{pol-immemoriam}
Let $0 < \epsilon <1/4$. Let $\varphi \in \dk$ be  positive with 
$\vc \varphi\vd  =1$. If $C$ is  a Borel subset of $X$, then \[\lambda_{\varphi} (C)\notin
\left(\frac{1-\sqrt{1-4\epsilon}}{2},\frac{1+\sqrt{1-4\epsilon}}{2}\right).\]
\end{lem}

\begin{proof}
 Suppose, contrary to what we claim, that there is a Borel subset $C$
such that $
\left(1-\sqrt{1-4\epsilon} \right)/2 < \lambda_{\varphi} (C) < \left( 1+\sqrt{1-4\epsilon}\right)/2$. This implies that $\lambda_{\varphi} (C) ( 1 - \lambda_{\varphi} (C)) > \epsilon$ and, consequently, we can find
$\delta>0$ with 
$(\lambda_{\varphi} (C) -\delta)(1-  \lambda_{\varphi} (C) -\delta)>\epsilon$.

By the regularity of the measure, there exist two compact subsets,
$K_1$ and $K_2$, such that $K_1\subset C$ and
$K_2\subset X\setminus C$ and, furthermore,
$\lambda_{\varphi} (K_1) > \lambda_{\varphi} (C)-\delta$ 
and
$\lambda_{\varphi} (K_2) > 1-\lambda_{\varphi} (C)-\delta$.

On the other hand, let us choose two disjoint open subsets $U$ and
$V$ of $X$ such that $K_1\subset U$ and $K_2\subset V$. By
Urysohn's lemma, we can find two functions $f_1$ and $f_2$ in
$C(X)$ such that $0\le f_1 \le 1$, $0\le f_2 \le 1$, $f_1\equiv 1$
on $K_1$, $f_2\equiv 1$ on $K_2$, ${\rm supp} (f_1)\subset U$ and
${\rm supp}(f_2)\subset V$. Clearly, $f_1 f_2 \equiv 0$ and
\[\varphi(f_i)=\int_X f_i d\lambda_{\varphi}  \ge \lambda_{\varphi}(K_i)\]
for $i = 1, 2$. 
Besides, $\vc f_1\vd_{\infty}=\vc f_2\vd_{\infty}=1$. However,
\[\va \varphi(f_1) \vb \va \varphi(f_2)\vb\ge
(\lambda_{\varphi} (C)-\delta)((1-\lambda_{\varphi} (C)-\delta)>\epsilon,\] which
contradicts the $\epsilon$-disjointness preserving property of $\varphi$, and we are done.
\end{proof}

If $\varphi \in C(X)'$, then let us  define
\[\va \varphi \vb (f):=\int_{X}fd \va\lambda_{\varphi}  \vb = 
 \int_{X} f  \overline{ \frac{ d \lambda_{\varphi}}{d \va\lambda_{\varphi}  \vb  }  }  d \lambda_{\varphi}   \]
 for every $f \in C(X)$.

\begin{lem}\label{sintelefarri}
Given $\varphi \in C(X)'$, $\va \varphi \vb $ is a positive linear functional on  $ C(X)$
with $\vc \va \varphi \vb\vd  = \vc \varphi \vd$.
Moreover, if
$\epsilon >0$ and  $\varphi \in \dk$, then $\va \varphi \vb \in \dk$ and $\lambda_{\va \varphi \vb} = \va \lambda_{\varphi} \vb$.
\end{lem}

\begin{proof}
The first part is apparent. As for the second part, using Lusin's Theorem (see \cite[p. 55]{Ru}), we can find a sequence $\pl k_n \pr$ in $C(X)$ such that $$\lim_{n \ra \infty} \int_X \va k_n - 
\overline{ \frac{ d \lambda_{\varphi}}{d \va\lambda_{\varphi}  \vb  }  } \vb  d \va \lambda_{\varphi} \vb= 0,$$
and $\vc k_n \vd_{\infty} \le 1$ for every $n \in \mathbb{N}$. This implies that, for all $f \in C(X)$, $\va \varphi \vb (f) = \lim_{n \ra \infty} \varphi \pl f k_n \pr$, and we can easily deduce that $\va \varphi \vb $
 is $\epsilon$-disjointness preserving. It is also clear that $\lambda_{\va \varphi \vb} = \va \lambda_{\varphi} \vb$.
\end{proof}

\begin{lem} \label{L1}
Let $0 < \epsilon <1/4$. Let $\varphi \in \dk$,  
$\vc \varphi\vd  =1$. Then there exists $x \in X$ with
\[\va \lambda_{\varphi} (\{x\}) \vb \ge \sqrt{1-4\epsilon}.\]
Furthermore, if $0 < \epsilon <2/9$, then there exists a unique $x\in X$ with
\[\va \lambda_{\varphi} (\{x\}) \vb \ge \frac{1+\sqrt{1-4\epsilon}}{2}.\]
\end{lem}

\begin{proof} 
Let $0 < \epsilon <1/4$. We prove the result  first for positive functionals. Suppose that for every $x \in X$, 
$\lambda_{\varphi} (\{x\})<   \sqrt{1-4\epsilon}$. For each $x \in X$, take an open neighborhood $U(x)$ of $x$ with 
$\lambda_{\varphi} \pl U(x) \pr<   \sqrt{1-4\epsilon}$. Since $X$ is compact,  we can find $x_1, x_2,  \ldots, x_n$ in 
$X$ such that $ X = U (x_1) \cup U(x_2) \cup \cdots \cup U(x_n)$. Let $r_1:= \lambda_{\varphi} (U(x_1))$, 
$r_2:= \lambda_{\varphi} (U(x_1) \cup U(x_2))$, \ldots, $r_n:= \lambda_{\varphi} (U(x_1) \cup U(x_2) \cup \cdots \cup U(x_n)) $, 
and suppose without loss of generality that $r_1 <r_2 < \cdots <r_n =1$. By Lemma~\ref{pol-immemoriam}, 
 $r_1 \le \pl 1-\sqrt{1-4\epsilon} \pr/2$, and we can take 
$i_0 := \max \tl i : r_i \le \pl 1-\sqrt{1-4\epsilon}\pr /2 \tr$. We then see that $r_{i_{0+1}}$  belongs to 
$\pl \pl 1-\sqrt{1-4\epsilon} \pr  /2, \pl 1+\sqrt{1-4\epsilon} \pr / 2\pr$, against Lemma~\ref{pol-immemoriam}. This proves the 
first part of the lemma for positive functionals. 

If $\varphi$ is not positive, then we use Lemma~\ref{sintelefarri}, and from the above paragraph we have that there exists  $x \in X$ such that
\[\va \lambda_{\varphi}   (\{x\}) \vb = \va \lambda_{\varphi}  \vb (\{x\})\ge \sqrt{1-4\epsilon}.\]

As for the second part, if follows immediately from Lemma~\ref{pol-immemoriam} and the fact that   $\pl 1-\sqrt{1-4\epsilon}\pr /2 < \sqrt{1-4\epsilon}$ for $0< \epsilon < 2/9$. Finally, if there exist two different points $x_1 , x_2$ such that
$\va \lambda_{\varphi} (\{x_i\}) \vb \ge \pl 1+\sqrt{1-4\epsilon}\pr / 2$ ($i =1,2$), then
$\va \lambda_{\varphi} \vb (\{x_1,x_2\})\ge 1+\sqrt{1-4\epsilon}>1 $, against our assumptions. This completes the proof.
\end{proof}

\begin{lem} \label{llata}
Let $0 < \epsilon <1/4$. Let $\varphi \in \dk$,  $\vc \varphi\vd  =1$. Then
there exists  $x \in X$ with
\[\vc \varphi- \lambda_{\varphi} (\{x\})\delta_{x}\vd   \le 1 - \sqrt{1-4\epsilon}.\]
Furthermore, if  $0 < \epsilon <2/9$, then
there exists a unique $x\in X$ with
\[\vc \varphi- \lambda_{\varphi} (\{x\})\delta_{x}\vd   \le \frac{1-\sqrt{1-4\epsilon}}{2}.\]
\end{lem}

\begin{proof} It is easy to see that  
\begin{eqnarray*}
 1 &=& \vc \varphi \vd \\ &=& \vc \lambda_{\varphi} (\{x\})\delta_{x} \vd + \vc \varphi- \lambda_{\varphi} (\{x\})\delta_{x}\vd \\
&=& \va \lambda_{\varphi} (\{x\}) \vb + \vc \varphi- \lambda_{\varphi} (\{x\})\delta_{x}\vd , 
\end{eqnarray*}
and the conclusion follows from Lemma~\ref{L1}.
\end{proof}

\begin{cor}\label{L3}
Let $0 < \epsilon <1/4$. Suppose that $\varphi \in \dk$ and
\[2 \sqrt{\epsilon}<\vc \varphi \vd  \le 1.\]
Then there exists $x \in X$ such that
\[
\vc \varphi - \lambda_{\varphi}  (\{x\})\delta_{x}\vd   \le
\vc \varphi \vd  -\sqrt{\vc \varphi \vd^2-4\epsilon}.
\]
Furthermore, if $0 < \epsilon <2/9$ and $\sqrt{9\epsilon / 2}<\vc \varphi \vd  \le 1$, then there exists a unique $x \in X$ such that 
\[
\vc \varphi - \lambda_{\varphi}  (\{x\})\delta_{x}\vd   \le
\frac{\vc \varphi \vd  -\sqrt{\vc \varphi \vd^2-4\epsilon}}{2}.
\]
\end{cor}

\begin{proof}
Let $0 < \epsilon < 1/4$. It is apparent that $\varphi/ \vc \varphi\vd  $ has norm $1$ and is
$\epsilon/ \vc \varphi \vd^2$-disjointness preserving. Besides
$\epsilon/\vc \varphi \vd^2 < \epsilon / \left( 2 \sqrt{\epsilon}\right)^2 = 1/4.$

Hence, by Lemma \ref{llata}, there exists  $x \in X$ with
\[\vc  \frac{\varphi}{\vc \varphi\vd  } - \lambda_{\frac{\varphi}{\vc \varphi \vd  }} (\{x\})\delta_{x}\vd   \le
1-\sqrt{1-\frac{4\epsilon}{\vc  \varphi \vd^2}}\]
and we are done. The proof of the second part is similar.
\end{proof}

\begin{cor}\label{zapatillas1966}
Let $0 < \epsilon <2/9$. Suppose that $\varphi  \in \dk$ and $ \sqrt{9\epsilon/2}<\vc \varphi \vd  \le 1$, and that $x \in X$ is the point given in Corollary~\ref{L3}.
Then
\[ \sqrt{\vc \varphi \vd^2-4\epsilon} \le \left| \varphi (f) \right| \]
whenever $f \in C(X)$ satisfies $\va f(x)\vb =1 = \vc f \vd_{\infty}$.
\end{cor}

\begin{proof}
Let $f \in C(X)$ be such that $\va f(x)\vb =1 = \vc f \vd_{\infty}$.
By Corollary \ref{L3}, we have that
\[\left| \left| \varphi (f) \right| -  \left| \lambda_{\varphi}  (\{x\}) \right| \right| \le  \left|(\varphi- \lambda_{\varphi}  (\{x\})\delta_{x}) (f)\right| \le
\frac{\vc \varphi\vd  -\sqrt{\vc \varphi\vd  ^2-4\epsilon}}{2}.\] Hence, by applying
Lemma \ref{L1}
\begin{eqnarray*}
 \left| \varphi (f) \right| &\ge& \left| \lambda_{\varphi}  (\{x\}) \right| - \frac{\vc \varphi\vd  -\sqrt{\vc \varphi\vd  ^2-4\epsilon}}{2}\\
&\ge& \frac{\vc \varphi\vd   + \sqrt{\vc \varphi\vd^2-4\epsilon}}{2}- \frac{\vc \varphi\vd  -\sqrt{\vc \varphi\vd^2-4\epsilon}}{2}\\
&=& \sqrt{\vc \varphi\vd^2-4\epsilon}.
\end{eqnarray*}
\end{proof}

\section{The sequence $(\omega_n)$ for continuous linear functionals}\label{ninguno}

Recall that we have defined, for each $n \in \mathbb{N}$, 
 $$\omega_n := \frac{n^2 -1}{4 n^2}$$ 
and
\[\mathbb{A}_n: = \ql \omega_{2n-1}, \omega_{2n+1} \pr . \] 

The precise statement of the results in this section depends heavily  on the number $n$ such that $\epsilon \in \mathbb{A}_n$, and on the cardinality of $X$.

\smallskip

Suppose that $X$ is a finite set of $k$ elements, and that $\varphi \in C(X)'$ has norm $1$. Then it is immediate
 that there exists a point $x \in X$ with $\va \lambda_{\varphi} (\{x\}) \vb \ge 1/k$. We next see that this result can be sharpened when $k$ is even and $\varphi \in \dk$, and also when $X$ has "many" elements (being finite or infinite).

\begin{prop}\label{ksis}
Let $0 < \epsilon < 1/4$. 
Suppose that $X$ is a finite set of cardinal  $k \in 2 \mathbb{N}$.
If $\varphi \in \dk$ and $\vc  \varphi \vd =1$, then there exists  
$x \in X$ such that $$\va \lambda_{\varphi} (\{x\}) \vb \ge \frac{1 + \sqrt{1-4 \epsilon}}{k}.$$
\end{prop}

\begin{proof}
By Lemma~\ref{sintelefarri}, we can assume without loss of generality that $\varphi$ is positive. Suppose that $k=2m$, $m \in \mathbb{N}$. 
Notice that there cannot be $m$ different points $x_1, \ldots, x_m \in X$ with \[\lambda_{\varphi} (\{x_i\}) \in \pl \frac{1 - \sqrt{1 -4 \epsilon}}{k} , \frac{1 + \sqrt{1 -4 \epsilon}}{k}\pr\] for every $i \in \{1, \ldots, m\}$, because otherwise $$\lambda_{\varphi}(\{x_1, \ldots, x_m \}) \in \pl \frac{1 - \sqrt{1 -4 \epsilon}}{2} , \frac{1 + \sqrt{1 -4 \epsilon}}{2}\pr,$$ against Lemma~\ref{pol-immemoriam}. This implies that there exist at least $m+1$ points whose measure belongs to $$\ql 0, \frac{1 - \sqrt{1 -4 \epsilon}}{k} \qr \cup \ql \frac{1 + \sqrt{1 -4 \epsilon}}{k}, 1 \qr . $$ Suppose that at least $m$ different points $x_1, \ldots, x_m \in X$ satisfy $\lambda_{\varphi}(\{x_i\}) \le \pl 1 - \sqrt{1 -4 \epsilon} \pr /k$. Then $\lambda_{\varphi}(\{x_1, \ldots, x_m\}) \le \pl 1 - \sqrt{1 -4 \epsilon} \pr / 2$, and consequently  we have that $\lambda_{\varphi}( X \setminus \{x_1, \ldots, x_m\}) \ge \pl 1 + \sqrt{1 -4 \epsilon} \pr / 2$. Since $X \setminus \{x_1, \ldots, x_m\}$ has $m$ points, this obviously implies that there exists $x \in X \setminus \{x_1, \ldots, x_m\}$ with $\lambda_{\varphi}(\{x\}) \ge  \pl 1 + \sqrt{1 -4 \epsilon} \pr / k$, and  we are done.
\end{proof}

 \begin{prop}\label{martopaz}
Let $0< \epsilon < 1/4$,  and let $n \in \mathbb{N}$ be such that $\epsilon \in \mathbb{A}_n $.
 Suppose that $\ca X \ge 2n$. 
If $\varphi \in \dk$ and $\vc  \varphi \vd =1$, then there exists  
$x \in X$ such that $$\va \lambda_{\varphi} (\{x\}) \vb \ge \frac{1 + \sqrt{1-4 \epsilon}}{2n}.$$
\end{prop}

\begin{proof}
Let $D:= \{x \in X : \va \lambda_{\varphi} (\{x\}) \vb >0\}$. It is clear that $D$ is a countable set, and by  Lemma~\ref{L1} it is nonempty. Let $\mathbb{M} := \{1, \ldots, m\}$ if the cardinal of $D$ is $m \in \mathbb{N}$, and let $\mathbb{M} := \mathbb{N}$ otherwise. It is obvious that we may assume that $D= \{x_i : i \in \mathbb{M      }\}$ and that $\va \lambda_{\varphi} (\{x_{i+1} \}) \vb \le \va \lambda_{\varphi} (\{x_{i} \}) \vb $ for every $i$.

Next let \[\mathbb{J} := \tl  j \in \mathbb{M}: \sum_{i=1}^j \va \lambda_{\varphi} (\{x_{i} \}) \vb < \frac{1}{2} \tr  \]
and  \[R:= \sum_{i \in \mathbb{J}} \va \lambda_{\varphi} (\{x_{i} \}) \vb.\]

We have that $R \le 1/2$, and by Lemma~\ref{pol-immemoriam} applied to the functional associated to  $\va \lambda_{\varphi} \vb$,  we get $R <1 /2$. 
Take any open subset  $U$   of $X$ containing all $x_i$, $i \in \mathbb{J}$, such that $\va \lambda_{\varphi} \vb (U ) < 1/2$, that is, $\va \lambda_{\varphi} \vb (U ) \le \pl 1 - \sqrt{1-4 \epsilon} \pr /2$, and suppose that $\va \lambda_{\varphi} (\{x\}) \vb < \sqrt{1-4 \epsilon} $ for every $x \notin U$. Then there exist open sets $U_1, \ldots, U_l$ in $X$, $l \in \mathbb{N}$, such that $X = U \cup U_1 \cup \cdots \cup U_l$ and $\va \lambda_{\varphi} \vb (U_i) < \sqrt{1-4 \epsilon} $ for every $i$. If we consider, for $i \in \{1, \ldots, l\}$,  $b_i := \va \lambda_{\varphi} \vb \pl U \cup \bigcup_{j=1}^i U_j \pr$, then we see that there must be an index $i_0$ with $$b_{i_0} \in \pl \frac{1 - \sqrt{1-4 \epsilon}}{2} , 
\frac{1 + \sqrt{1-4 \epsilon}}{2} \pr,$$ which goes against Lemma~\ref{pol-immemoriam}.

We deduce that there exists $j \in \mathbb{M}$, $j \notin \mathbb{J}$, such that $\va \lambda_{\varphi} (\{x_{j} \}) \vb \ge \sqrt{1 - 4 \epsilon}$. By the way we have taken $D$, this implies that $\va \lambda_{\varphi} (\{x_{i} \}) \vb \ge \sqrt{1 - 4 \epsilon}$ for every $i \in \mathbb{J}$, and obviously $\mathbb{J}$ must be finite, say $\mathbb{J} = \{1 , \ldots, m_0\}$.

Let us see now that $m_0 \le n-1$.  We have that, since $\epsilon < \omega_{2n+1}$, then $\sqrt{1 - 4 \epsilon} > 1 / \pl 2n+1 \pr$, which implies that \[ n  \sqrt{1- 4 \epsilon} > \frac{1 - \sqrt{1- 4 \epsilon}}{2}.\]Consequently, if $m_0 \ge  n $, then we get 
\begin{eqnarray*}
R &=& 
\sum_{i=1}^{m_0} \va \lambda_{\varphi} (\{x_{i} \}) \vb \\ &\ge& n  \sqrt{1-4 \epsilon}
\\ &>& \frac{1 - \sqrt{1- 4 \epsilon}}{2},
\end{eqnarray*}
which is impossible, as we said above. We conclude that $m_0 \le  n-1 $. 

On the other hand, taking into account that  
\[\sum_{i =1}^{m_0 +1} \va \lambda_{\varphi} (\{x_{i} \}) \vb \ge \frac{1 + \sqrt{1 - 4 \epsilon}}{2},\]
we have that \[(m_0 +1) \va \lambda_{\varphi} (\{x_{1} \}) \vb \ge \frac{1 + \sqrt{1 - 4 \epsilon}}{2},\]
which implies that \[n \va \lambda_{\varphi} (\{x_{1} \}) \vb \ge \frac{1 + \sqrt{1 - 4 \epsilon}}{2}.\]
As a consequence we get $$\va \lambda_{\varphi } ( \{x_1\}) \vb \ge \frac{1 + \sqrt{1-4 \epsilon}}{2n},$$ and we are done.
\end{proof}

\begin{proof}[Proof of Theorem~\ref{martakno}]
Let us show the first part. By Propositions~\ref{ksis} (see also comment before it) and  \ref{martopaz}, there exists 
$x \in X$ with $\va \lambda_{\varphi} (\{x\}) \vb \ge 1-  o_X (\epsilon)$. If we define $\psi := \lambda_{\varphi} (\{x\}) \delta_x$, then we are done.

\smallskip
Let us now prove the second part. Suppose that $\epsilon$ belongs to $\mathbb{A}_n $, $n \in \mathbb{N}$. It is clear that this fact implies that 
$(2 n -1) \sqrt{1- 4 \epsilon} \le 1$.

If $\ca X \ge 2 n$, then  we can pick $2n$  distinct points $x_1, x_2, \ldots, x_{2n}$ in $X$, and  define the  map $\varphi \in C(X)'$ as  \[\varphi  := \frac{1 + \sqrt{1-4 \epsilon}}{2n} \pl \sum_{i=1}^{2n-1} \delta_{x_i} \pr + \frac{1 - (2n -1) \sqrt{1-4 \epsilon}}{2n} \hspace{.03in} \delta_{x_{2n}} . \] It is easy to see that $\varphi$ satisfies all the requirements.

To study the cases when $\ca X < 2n$, put  $X := \{x_1, \ldots, x_k\}$. 
Suppose first that $k$ is even. Since  $(2n -1) \sqrt{1- 4 \epsilon} \le 1$, we have $(k-1)   \sqrt{1- 4 \epsilon} < 1$. We can easily  see that if 
 we 
 define the map $\varphi$ as
 \[\varphi := \frac{1 + \sqrt{1- 4 \epsilon}}{k} \pl \sum_{i=1}^{k-1} \delta_{x_i} \pr + \frac{1 -(k-1) \sqrt{1- 4 \epsilon}}{k} \hspace{.03in} \delta_{x_k}  ,\] 
then we are done.

Suppose finally that $k$ is odd. It is clear that if we define 
\[\varphi :=  \frac{1}{k} \pl \sum_{i=1}^{k}  \delta_{x_i} \pr,\]
then $\varphi$ is a norm one element of $C(X)'$, and is $\omega_k$-disjointness preserving, which implies that it is $\epsilon$-disjointness preserving. It is also easy to see that $\vc \varphi - \psi \vd \ge 1- 1/k$ for every weighted evaluation functional $\psi$ on $C(X)$.
\end{proof}

\section{How close. The general case: Proofs}\label{paulus-15nov07}

Let  $0 < \epsilon
< 2/9$, and let  $T: C(X) \longrightarrow C(Y)$ be a norm one $\epsilon$-disjointness preserving operator. 
If we take any $y\in Y_{\sqrt{9\epsilon/2}}$, then
$T_{y} / \vc T_{y}\vd$  is a norm one 
$\epsilon/\vc T_y\vd^2$-disjointness preserving operator with
\[\frac{\epsilon}{\vc T_y\vd^2}<\frac{\epsilon}{\frac{9\epsilon}{2}}
=\frac{2}{9}.\]
 
By Lemma \ref{L1}, there exists a unique $x_y\in X$ such that
$\va \lambda_{T_y}(\{x_y\}) \vb > \vc T_y\vd  /2$.
Thus, we can define a map $h_T: Y_{\sqrt{9\epsilon/2}} \longrightarrow X$, in such a way that 
$\va \lambda_{T_y} (\{h_T (y) \}) \vb > \vc T_y \vd/2$ for each $y \in Y_{\sqrt{9\epsilon/2}}$.

These fact can be summarized in the following lemma.

\begin{lem}\label{nak}
Let  $0 < \epsilon
< 2/9$, and let $T \in \dg$ with $\vc  T\vd =1$. If $y \in Y_{\sqrt{9\epsilon/2}}$, then $$\va \lambda_{T_{y}} \vb (\{h_T (y)\}) \ge
 \frac{\vc T_{y}\vd  + \sqrt{\vc T_{y}\vd^2-4\epsilon}}{2}.$$
\end{lem}

\begin{prop}\label{lg}
Let  $0 < \epsilon
< 2/9$, and let $T \in \dg$ with $\vc  T\vd =1$. Then the map $h_T$  is continuous.
\end{prop}

\begin{proof}
We will check the continuity of this map at every point.
To this end, fix $y_0\in  Y_{\sqrt{9\epsilon/2}} $ and let $U(h_T (y_0))$ be an open neighborhood  of $h_T (y_0)$. We have to find an open neighborhood $V(y_0)$ of $y_0$ such that $h_T (V(y_0))\subset U(h_T (y_0))$.

By regularity,
 there exists an open
neighborhood $U'(h_T (y_0))\subset U(h_T (y_0))$ of $h_T (y_0)$ such that
\[\va \lambda_{T_{y_0}} \vb (U'(h_T (y_0)))- \va \lambda_{T_{y_0}} \vb (\{h_T (y_0)\})< \frac{\sqrt{\vc T_{y_0}\vd^2-4\epsilon}}{2}.\]

Let $f_0\in C(X)$ with $0\le f_0\le 1$, $f_0(h_T (y_0))=1$, and 
${\rm   supp}(f_0)\subset U'(h_T (y_0))$.

We will now check that $\va (Tf_0)(y_0) \vb > \sqrt{\epsilon}$.
To this end, we proceed as follows:
 \begin{eqnarray*}
\va (Tf_0)(y_0) \vb
&=& 
\va\int_X f_0 d \lambda_{T_{y_0}} \vb
\\ &=&\va \int_{\{h_T(y_0)\}}f_0d\lambda_{T_{y_0}}+\int_{U'(h_T(y_0))\setminus \{h_T(y_0)\}}f_0d\lambda_{T_{y_0}} \vb
\\ &\ge&
\va \int_{\{h_T(y_0)\}}f_0d\lambda_{T_{y_0}} \vb - \int_{U'(h_T(y_0))\setminus \{h_T(y_0)\}}f_0 d \va    \lambda_{T_{y_0}} \vb
\\ &\ge&
\va f_0(h_T(y_0)) \vb  \va \lambda_{T_{y_0}}(h_T(\{y_0\})) \vb- \va \lambda_{T_{y_0}} \vb (U' (h_T(y_0))\setminus
\{h_T(y_0)\})
\\ &>&
\frac{\vc T_{y_0}\vd  +\sqrt{\vc T_{y_0}\vd^2-4\epsilon}}{2}- \frac{\sqrt{\vc T_{y_0}\vd^2-4\epsilon}}{2},
\end{eqnarray*}
and as a consequence, we see that
\[\va (Tf_0)(y_0) \vb > \frac{\vc T_{y_0}\vd  }{2}> \sqrt{9\epsilon/8} >\sqrt{\epsilon},\]
as was to be checked.

Let us now define
\[V(y_0):=\left\{y\in Y: \va (Tf_0)(y) \vb >\sqrt{\epsilon}\right\}\cap Y_{\sqrt{9\epsilon/2}}.\]
We will check that, if $y_1\in V(y_0)$, then $h_T(y_1)\in {\rm
supp}(f_0)$. Assume, contrary to what we claim, that $h_T(y_1)\notin
{\rm supp}(f_0)$. Then there exist
an
open set $U'(h_T(y_1))$ and a function $f_1\in C(X)$ such that ${\rm
supp}(f_1)\cap {\rm supp}(f_0)=\emptyset$, $0\le f_1\le 1$,
$f_1(h_T(y_1))=1$ and ${\rm supp}(f_1)\subset U'(h_T(y_1))$ with
\[\va \lambda_{T_{y_1}} \vb (U'(h_T(y_1))- \va \lambda_{T_{y_1}} \vb (\{h_T(y_1)\}) < \frac{\sqrt{\vc T_{y_1}\vd^2-4\epsilon}}{2}\]
As above,
\[\va (Tf_1)(y_1) \vb >  \sqrt{\epsilon}.\]
Hence,
\[\vc (Tf_1) ( Tf_0) \vd_{\infty}  \ge \va (Tf_1)(y_1) \vb \va (Tf_0)(y_1) \vb > \epsilon,\]
which contradicts the $\epsilon$-disjointness preserving property of $T$.
Summing up, $h_T$ is continuous.
\end{proof}

\begin{lem}\label{cv}
Let  $0 < \epsilon
< 2/9$, and let $T \in \dg$ with $\vc  T\vd =1$. If $t \in [0,1]$ and $y \in Y_{\sqrt{9\epsilon/2}}$, then
\[\va (Tf) (y) - t (T{\bf 1}) (y) f(h_T(y)) \vb \le  \vc T_y \vd  - t \sqrt{\vc T_y\vd^2-4\epsilon}\] for every $f \in C(X)$ with $\vc f \vd_{\infty} \le 1$.
\end{lem}

\begin{proof}
Let $A_y:= \lambda_{T_y} (\{h_T(y)\})\delta_{h_T(y)}$.
It is easy to check that, since $(\lambda_{T_y} -\lambda_{A_y})
(\{h_T(y) \}) =0$, then $\vc T_y\vd   = \vc T_y -A_y \vd   + \vc A_y \vd  $.
Furthermore, as by Lemma~\ref{nak},
\[\vc A_y \vd   \ge  \frac{\vc T_y\vd  +\sqrt{\vc T_y\vd^2-4\epsilon}}{2},\]
we deduce
\begin{equation*}
- \sqrt{\vc T_y\vd^2-4\epsilon} \ge \vc T_y\vd  - 2 \vc A_y \vd.
\end{equation*}

As a consequence, for $f\in C(X)$ with $\vc f\vd_{\infty}  \le 1$, we have
\begin{eqnarray*}
\va (Tf)(y)- t (T{\bf 1}) (y) f(h_T(y)) \vb &\le& \va T_y f- A_y f \vb + \va A_y f - t A_y f \vb  + \\& & \hspace{0.1in} + t \va A_y
\widehat{f(h_T(y))} - T_y \widehat{f(h_T(y))} \vb
\\&\le&
\vc T_y -A_y\vd   + (1 - t) \vc A_y \vd   + t \vc T_y -A_y\vd  
\\ &=&
(1 + t) (\vc T_y\vd  - \vc  A_y\vd ) + (1- t ) \vc  A_y\vd 
\\ &=&
\vc  T_y\vd  + t (\vc  T_y \vd  - 2 \vc A_y \vd )
\\&\le& \vc T_y \vd  - t \sqrt{\vc T_y\vd^2-4\epsilon},
\end{eqnarray*}
and we are done.
\end{proof}

\begin{lem}\label{pollo}
Let $0 < \epsilon < 1/4$. The function $\gamma : [2 \sqrt{\epsilon}, 1] \longrightarrow \mathbb{R}$, defined as  $\gamma (t) := t - \sqrt{t^2 -4 \epsilon}$ is strictly decreasing and  bounded above by $2 \sqrt{\epsilon}$.
\end{lem}

\begin{proof}[Proof of Theorem~\ref{rz}]
We  fix $\delta_0 \in \pl 0,  \epsilon \pl 1- \sqrt{17 \epsilon /2}\pr \pr$ and, for each $n \in \mathbb{N} $, set $\delta_n := \delta_0 2^{-n}$. Also we define $D_n := Y_{\sqrt{17 \epsilon /2}+\delta_n}$ for $n \in \mathbb{N} \cup \{0\}$, which is nonempty for every $n$.

We easily deduce from Corollary~\ref{zapatillas1966} that $\va (T{\bf 1}) (y) \vb \ge \sqrt{9 \epsilon/2} + \delta_n$ for every $n \in \mathbb{N}$ and $y \in \cl D_n $. Consequently each $\cl D_n $ is contained in $Y_{\sqrt{9\epsilon/2}}$, so there exists a function $\alpha_n \in
C(Y)$ such that $0\le \alpha_n \le 1$,
\[\alpha_n \left(\cl D_n  \right)\equiv 1\]
and
\[{\rm supp}(\alpha_n)\subset Y_{\sqrt{\frac{9\epsilon}{2}}}.\]
Let us define $\alpha : Y \ra \mathbb{K}$ as \[\alpha := \sum_{n=0}^{\infty} \frac{\alpha_n}{2^{n+1}}.\]It is clear that $\alpha$ is continuous, $\vc \alpha \vd_{\infty} =1$, $c(\alpha)\subset Y_{\sqrt{9\epsilon/2}}$, $\alpha \left( D_0 \right) \equiv 1$, and $\alpha \ge 1/2^n$ on $D_n$ for each $n \in \mathbb{N}$. Finally define a weighted
composition map $S$ as \[(S f)(y):=\alpha(y) (T {\bf 1})(y)f(h_T(y))\]
for all $f\in C(X)$ and $y\in Y$.

\medskip
We will now check that
\[\vc T- S\vd  \le \sqrt{\frac{17 \epsilon}{2}}.\]
Fix any  $f\in C(X)$
with $\vc f\vd_{\infty}=1$.

Let us first study  the case of $y \in Y$ satisfying
$\vc T_y\vd  \le \sqrt{9\epsilon/2}$. Since $c (\alpha)\subset Y_{\sqrt{9\epsilon/2}}$, in this case   we have
\[\va (Tf)(y)-( S f)(y) \vb = \va (Tf)(y) \vb \le \sqrt{\frac{9\epsilon}{2}} .\]

Next, consider the remaining case $\sqrt{9\epsilon/2}<
\vc T_y\vd  \le 1$.
By Lemma~\ref{cv} we know that 
\[\va (Tf)(y)-(Sf)(y) \vb  \le\vc T_y \vd  - \alpha(y) \sqrt{\vc T_y\vd^2-4\epsilon}\] for every $y \in Y_{\sqrt{9\epsilon/2}}$.

We immediately deduce that $\va (Tf)(y)-(Sf)(y) \vb \le  \sqrt{17 \epsilon /2}$ for every $y$ with $\sqrt{9\epsilon/2}<
\vc T_y\vd  \le \sqrt{17 \epsilon /2}$. On the other hand, for   $y\in D_0$, we have $\alpha (y) =1$, so \[\va (Tf)(y)-(Sf)(y) \vb \le \vc T_y\vd  -\sqrt{\vc T_y\vd^2-4\epsilon} \le
2\sqrt{\epsilon}\] by Lemma~\ref{pollo}.

Finally, if $\sqrt{17 \epsilon /2} < \vc T_y\vd  \le \sqrt{17 \epsilon /2} + \delta_0$, then there exists $n \in \mathbb{N}$ such that $y \in D_n \setminus D_{n-1}$, that is, $\sqrt{17 \epsilon /2} + \delta_{n} < \vc T_y\vd  \le \sqrt{17 \epsilon /2} + \delta_{n-1}$.  Let us see that \begin{equation}\label{eqf}
\delta_{n-1} \le \alpha(y) \sqrt{\vc T_y\vd^2-4\epsilon}.
\end{equation}
Clearly, since we have chosen $\delta_0 < \sqrt{\epsilon}$, we know that 
\[2 \delta_0 < \sqrt{\vc T_y\vd^2-4\epsilon}.\]
Also,  by the definition of $\alpha$, we have $\alpha (y) \ge 1/2^n$, and the inequality \ref{eqf} follows. In this way we get 
\begin{eqnarray*}
\va (Tf)(y)-(Sf)(y) \vb &\le& \vc T_y \vd  - \alpha(y) \sqrt{\vc T_y\vd^2-4\epsilon} \\
&\le& \sqrt{\frac{17 \epsilon}{2}} + \delta_{n-1} - \alpha(y) \sqrt{\vc T_y\vd^2-4\epsilon} \\
&\le& \sqrt{\frac{17 \epsilon}{2}}.
\end{eqnarray*}

We conclude that $\vc T- S \vd \le \sqrt{17 \epsilon /2}$, as was to be proved.
\end{proof}

\section{How close. The general case: Examples}\label{pexample}

In this section we  first provide a sequence of examples  of  $2/9$-disjointness preserving operators of norm $1$, and then give  a related family of $2/17$-disjointness preserving operators. This will lead to an example of a norm one $2/17$-disjointness operator whose distance to every weighted composition map is at least $1$. We use  this to get an example, for each $\epsilon \in (0, 2/17)$, of an element of $\dg$ whose distance to $\wc$ is at least $\sqrt{17 \epsilon/2}$. This shows that
the bound given in Theorem~\ref{rz} is sharp.  All the sets involved in these examples are contained in $\mathbb{R}^3$.

\begin{ex}\label{deruhi-abiri}
{\em A special sequence $\pl R_n \pr$ of $2/9$-disjointness preserving operators of norm $1$.}

\medskip

For any two points $A, B$ in $\mathbb{R}^3$, let $\overline{AB}$  denote the segment joining them.
Also, given four points $A,B,C,D \in \mathbb{R}^3$, let  $\vee{ABCD}$ denote the union of the segments $\overline{AD}$, $\overline{BD}$ and $\overline{CD}$.

We will need some (closed) semilines,  all  contained in $z=0$, and  starting at the point $(0,0,0)$. First $l_A$ will be the semiline $x \ge 0$, $y=0$. 
We now take two other 
semilines, namely 
\begin{eqnarray*}
l_B &:=& \mathbf{rot} \pl l_A, \frac{2 \pi}{3} \pr \\
l_C &:=& \mathbf{rot} \pl l_B, \frac{2 \pi}{3} \pr ,
\end{eqnarray*} 
where for $G \subset \mathbb{R}^3$ and 
 $\theta \in [0, 2 \pi)$, $\mathbf{rot} (G, \theta)$ denotes the set obtained by rotating counterclockwise 
 $G$ an angle $\theta$ around the axis $z$.

Next consider the circles $S_1, S_1'$ and $S_1''$ centered at $(0,0,0)$ with radius $1, 2$, and $3$,  respectively, and contained in the plane $z=0$. 
For $E \in \{A, B, C\}$, we denote by $E_0, E_0'$ and $E_0''$ the points in the intersection of $l_E$ with $S_1, S_1'$ and $S_1''$, respectively.

Fix $\theta_0 := \pi/6$. For each  $n \in \mathbb{N}$ and $E=A, B, C$,  we now define a new semiline as $$m^E_n := \mathbf{rot} \pl l_E ,2 \pi -  \frac{\theta_0}{n} \pr.$$ We will use each of these semilines to obtain two new points, $E_n$ and $E_n'$, as the intersection of $m^E_n$ with $S_1$ and $S_1'$, respectively. That is,
if for $\mathbf{x} \in \mathbb{R}^3$ and $\theta \in [0,2 \pi)$, we write $\mathbf{rot} ( \mathbf{x}, \theta)$ meaning the point in $\mathbf{rot} ( \tl \mathbf{x} \tr, \theta)$, then 
$E_n := \mathbf{rot} ( E_0, 2 \pi -  \theta_0 /n)$
and
$E_n' := \mathbf{rot} ( E_0',  2 \pi -  \theta_0 /n),$
for $E=A, B, C$ and $n \in \mathbb{N}$.

We put  $D_0 := (0,0,0)$, and introduce two special points $D_1^n := (0,0,1/n^3)$, $D_2^n := (0,0,2/n^3)$ for each $n \in \mathbb{N}$. We also denote $D_0^n := D_0$ for all $n \in \mathbb{N}$.

Given $ n \in \mathbb{N}$, we start by considering the sets $W_n^0 := \vee{A_n B_0 C_0 D_0}$, $W_n^1 := \vee{A_0 B_n C_0 D^n_1}$, and
$W_n^2 := \vee{A_0 B_0  C_n D^n_2}$ (for the case $n=1$, see Figures~\ref{w0}, \ref{w1}, and \ref{w2}, respectively).
\begin{figure}[ht]
  \caption{The set $W_1^0$.}\label{w0}
  \begin{center}
    \includegraphics[scale=.3]{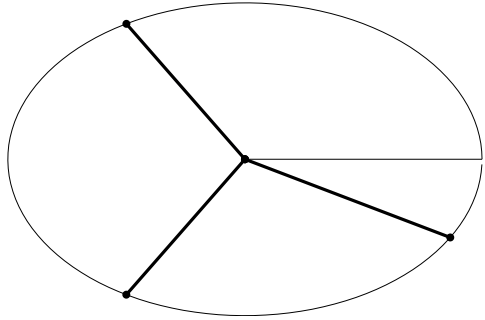}
  \end{center}
\end{figure}

\begin{figure}[ht]
  \caption{The set $W_1^1$}\label{w1}
  \begin{center}
    \includegraphics[scale=.3]{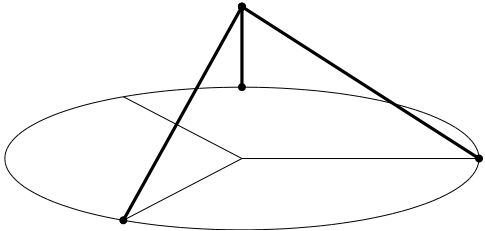}
  \end{center}
\end{figure}

\begin{figure}[ht]
  \caption{The set $W_1^2$}\label{w2}
  \begin{center}
    \includegraphics[scale=.3]{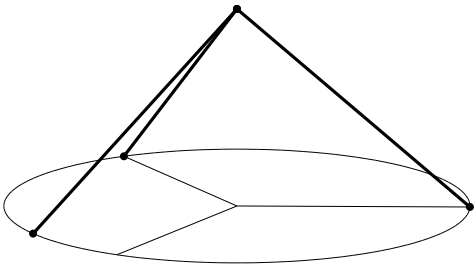}
  \end{center}
\end{figure}
 By the way we have taken all these points, we see
that for every $n$, the intersection of two different $W_n^i$ consists of one of the points $A_0, B_0, C_0$. In the same way, 
for $E=A, B, C$, each  $E_0$ belongs exactly to two sets $W_n^i$, and each  $E_n$ belongs to just one of them. Call $Z_n := W_n^0 \cup W_n^1 \cup W_n^2$ (see Figure \ref{z1} for the case $n=1$).

\begin{figure}[ht]
  \caption{The set $Z_1$}\label{z1}
  \begin{center}
    \includegraphics[scale=.3]{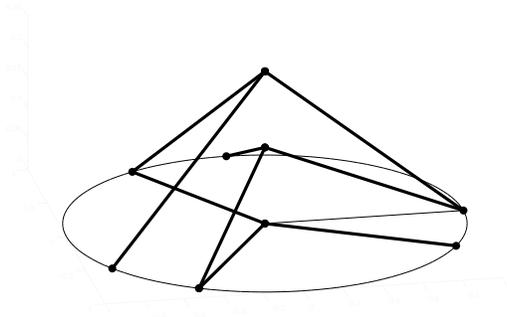}
  \end{center}
\end{figure}

Consider next a new set $W_0 := \vee{A_0 B_0 C_0 D_0}$, and (see Figure \ref{x0})$$X_0 := W_0 \cup \pl \bigcup_{E=A,B,C} \overline{E_0 E''_0} \pr .$$

\begin{figure}[ht]
  \caption{The set $X_0$}\label{x0}
  \begin{center}
    \includegraphics[scale=.3]{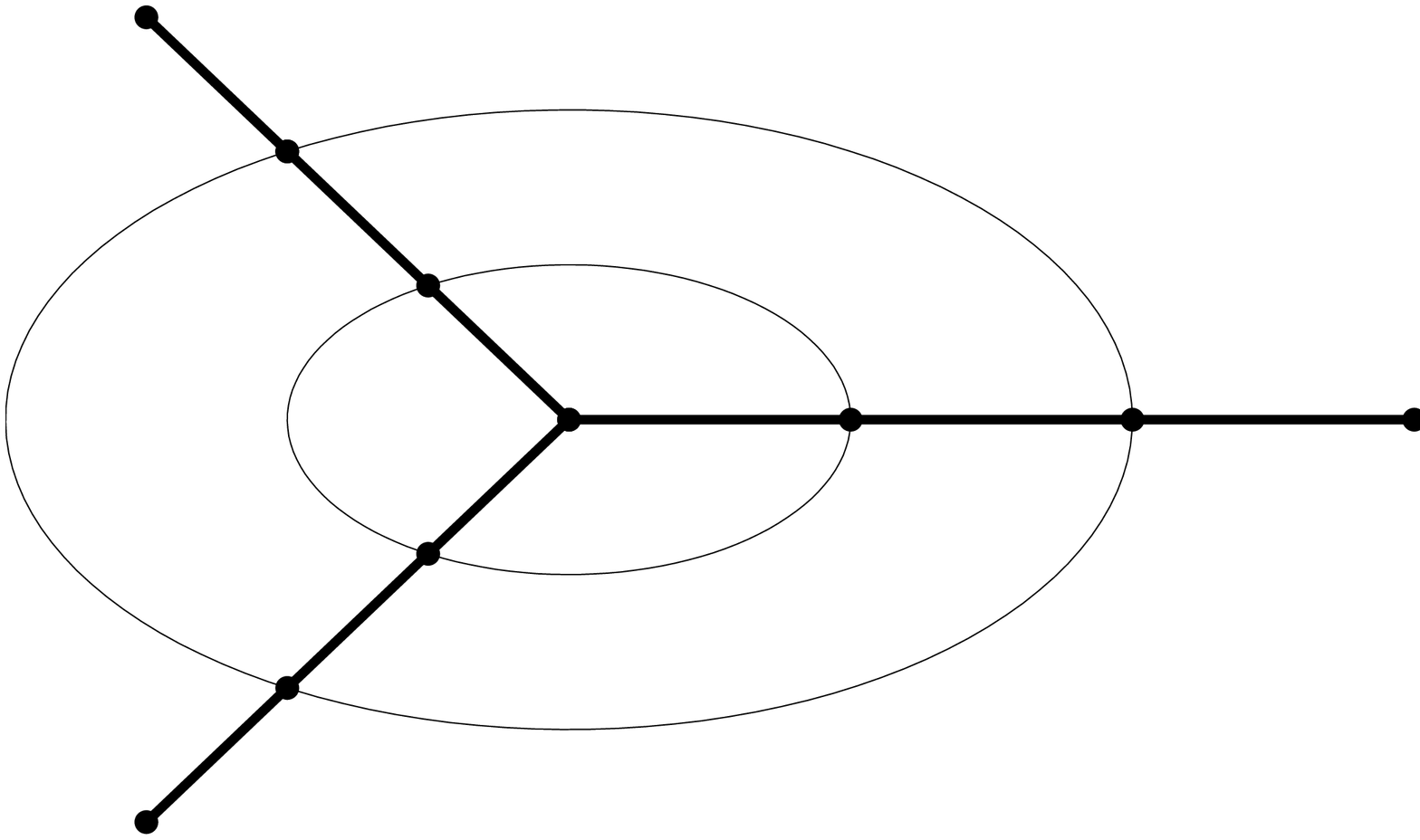}
  \end{center}
\end{figure} Also, for each $n \in \mathbb{N}$, define
$$X_n:= Z_n \cup \pl \bigcup_{E=A,B,C} \overline{E_0 E''_0} \pr \cup \pl \bigcup_{E=A,B,C} \overline{E_n E'_n} \pr$$(see Figures \ref{x1} and \ref{x1ab} for the case $n=1$)
\begin{figure}[ht]
  \caption{The set $X_1$}\label{x1}
  \begin{center}
    \includegraphics[scale=.3]{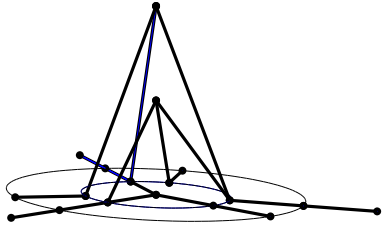}
  \end{center}
\end{figure} 

\begin{figure}[ht]
  \caption{Projection of $X_1$ on the plane $z=0$}\label{x1ab}
  \begin{center}
    \includegraphics[scale=.3]{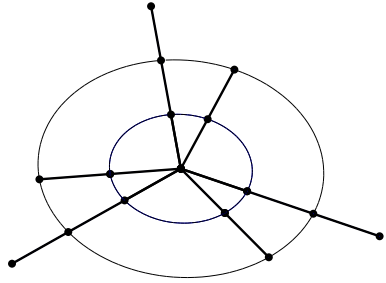}
  \end{center}
\end{figure}

\medskip

Our next step consists of introducing  a linear  and continuous operator $R_n: C(X_n) \ra C(X_0)$ for every $n \in \mathbb{N}$. 
Given any point $\mathbf{y} \in X_0$ and $f \in C(X_n)$, the definition of $R_n f$ at $\mathbf{y}$ will depend on whether  or not 
$\mathbf{y} \in W_0$.

To this end, for each $i=0, 1,2$ and $n \in \mathbb{N}$, we  will  define 
a map $h_n^i :W_0 \ra W_n^i$. Notice that, given $\mathbf{y} \in W_0$,      there exist $E \in \{A, B, C\}$ and 
$t \in [0,1]$ such that $\mathbf{y} = t E_0$. For such $\mathbf{y}$, we put  
$h_n^i ( \mathbf{y}) := D_i^n + t \pl E_{\mathbf{\mathbf{y}}} - D_i^n \pr $, where  $E_{\mathbf{y}} = E_0$ or $E_n$ (depending on whether $E_0 \in W_n^i$ or  $E_n \in W_n^i$). It is immediate to see 
that $h_n^i$ is indeed a surjective homeomorphism.

Suppose that $f \in C(X_n)$, and that $\mathbf{y} \in W_0$. Then
we define $$ (R_n f) (\mathbf{y}) := \frac{ f \pl h_n^1 (\mathbf{y}) \pr + f \pl h_n^2 (\mathbf{y}) \pr + f \pl h_n^3 (\mathbf{y}) \pr}{3}.$$

The definition of $R_n f$ at points of $X_0 \setminus W_0$ will be given in a different way. To do so we fix a continuous map $\zeta : [0,1] \ra [2/3,1]$ such that $\zeta (0) =2/3$ and $\zeta (1) = 1$.
For   $E=A,B,C$,  given $\mathbf{y} \in \overline{E_0 E'_0}$, there exists $t \in [0,1]$ such that $\mathbf{y} = E_0 + t \pl E_0' - E_0 \pr$. For such point, 
we set
\begin{eqnarray*}
 (R_n f) (\mathbf{y}) &:=&    \zeta(t) f (\mathbf{y}) + ( 1 - \zeta (t) )   f \pl E_n + t \pl E_n' - E_n \pr \pr. 
 \end{eqnarray*}
Finally, if $\mathbf{y}\in \overline{E'_0 E''_0}$, then define
$$ (R_n f ) (\mathbf{y}) := f(\mathbf{y}).$$

In particular we see that for  every $f \in C (X_n)$ and $E=A, B, C$, $(R_n f) (E_0) = 2 f\pl E_0 \pr /3 +  f\pl E_n \pr /3$, $(R_n f) (E'_0) =  f\pl E'_0 \pr $, and $(R_n f) (E''_0) =  f\pl E''_0 \pr $.

On the other hand, it is easy to check that $R_n$ is  linear and continuous, with $\vc R_n \vd = 1$, and that it is $2/9$-disjointness preserving.
\end{ex}

\begin{ex}\label{bebel}
{\em A special sequence $\pl T_n \pr$ of $2/17$-disjointness preserving operators of norm $1$.}

\medskip

We follow the same notation as in Example~\ref{deruhi-abiri}. To construct the new examples  take a continuous map $$\rho : X_0 \ra \ql \frac{3}{\sqrt{17}} , 1 \qr$$ such that $\rho \pl W_0 
\cup \bigcup_{E=A, B, C} \overline{E_0  E'_0} \pr  \equiv 3/\sqrt{17}$
and $\rho \pl \tl A''_0 , B''_0, C''_0 \tr \pr \equiv 1$. For each $n \in \mathbb{N}$ define $T_n : C(X_n ) \ra C(X_0)$ as $T_n f := \rho \cdot R_n f$ for every $f \in C(X_n)$.

Taking into account that $R_n$ is 
$2/9$-disjointness preserving, it is straightforward to see that each $T_n$  is $2/17$-disjointness preserving. On the other hand, it is immediate that it has norm $1$.

\medskip

Unfortunately, it is possible to construct a weighted composition map $S_n: C(X_n ) \ra C(X_0)$ whose distance to
$T_n$ is strictly less than $1$. This can be done as follows. Pick any $\epsilon >0$, and consider $a_n \in C(X_0)$,
$0 \le a_n \le 1$, 
such that $a_n \equiv 1$ on $\rho^{-1} \pl \ql 3/ \sqrt{17} + \epsilon, 1 \qr \pr$ and $a_n \equiv 0$ on $\rho^{-1} \pl \tl 3/ \sqrt{17} \tr \pr$. If we define $S_n : C(X_n) \ra C(X_0)$  as $(S_n f) (\mathbf{y}) := 
a_n (\mathbf{y}) \rho (\mathbf{y}) f (\mathbf{y})$ for every $f \in C(X_n)$ and $\mathbf{y} \in X_0$, then
$ (S_n f) (\mathbf{y})  = (T_n f) (\mathbf{y})$ when $y \in \rho^{-1} \pl \ql 3/ \sqrt{17} + \epsilon, 1 \qr \pr$. Thus  
 $\vc T_n - S_n \vd \le 3/ \sqrt{17} + \epsilon$.
\end{ex} 

Consequently, constructing an operator having the desired properties is more complicated. It will be done in the next example.

\begin{ex}\label{andalu}
{\em A $2/17$-disjointness preserving operator of norm $1$ whose distance to any weighted composition map is at least $1$.}

\medskip

We follow the notation given in Examples~\ref{deruhi-abiri} and \ref{bebel}. Also, for $n \in \mathbb{N}$, we put  $\mathbf{w}_n := (0,0, 1/n)$, and define 
\begin{eqnarray*}
X &:=& X_0 \cup \pl \bigcup_{n=1}^{\infty} \mathbf{w}_n + X_n  \pr \cup \pl \bigcup_{n=1}^{\infty} - \mathbf{w}_n + X_0  \pr ,\\
Y &:=& X_0 \cup \pl \bigcup_{n=1}^{\infty} \mathbf{w}_n + X_0  \pr \cup \pl \bigcup_{n=1}^{\infty} - \mathbf{w}_n + X_0  \pr 
\end{eqnarray*}

\smallskip

Related to the  $T_n$, we can introduce in a natural way new norm one $2/17$-disjointness preserving operators $T'_n : C \pl \mathbf{w}_n + X_n \pr \ra C \pl  \mathbf{w}_n + X_0 \pr$ as follows. First, given $\mathbf{z} \in \mathbb{R}^3$, denote by $\tau_{\mathbf{z}}$ the translation operator sending each $\mathbf{x} \in \mathbb{R}^3$ to $\mathbf{z} + \mathbf{x}$.
Then define $P_n :  C \pl \mathbf{w}_n + X_n \pr \ra C \pl  X_n \pr$ as $P_n f := f \circ \tau_{\mathbf{w}_n}$ for every $f \in C \pl \mathbf{w}_n + X_n \pr$, and $Q_n : C \pl   X_0 \pr \ra C \pl  \mathbf{w}_n + X_0 \pr$ as
$Q_n g := g \circ \tau_{- \mathbf{w}_n}$ for every $ g \in C \pl   X_0 \pr$. 
Finally put $T'_n := Q_n \circ T_n \circ P_n$.

\smallskip

Next we give a $2/17$-disjointness preserving linear and continuous operator $T: C(X) \ra C(Y)$ of norm $1$.
In the process of definition, as well as in the rest of this example, the restrictions of functions $f \in C(X)$
to subspaces of $X$ will be also denoted by $f$.
Take any $f \in C(X)$.
If $\mathbf{y} \in X_0$, then we define $$(Tf) (\mathbf{y}) := \rho (\mathbf{y}) f(\mathbf{y}).$$ 
Also,  for $n \in \mathbb{N}$ and $\mathbf{y} \in X_0$, we put 
$$(Tf) (-\mathbf{w}_n + \mathbf{y}) := \pl \frac{1}{2} + \frac{\rho (\mathbf{y})}{2} \pr f(-\mathbf{w}_{2n} + \mathbf{y}) - \pl \frac{1}{2} - \frac{\rho (\mathbf{y})}{2} \pr f(-\mathbf{w}_{2n -1} + \mathbf{y}) ,$$ 
and
$$(Tf) ( \mathbf{w}_n + \mathbf{y} ) :=   \pl T'_n f \pr (\mathbf{\mathbf{w}_n + y}) .$$

It is easy to check that $\vc T \vd =1$, and we can use the fact that $0 \le \pl 1 - \rho(\mathbf{z})^2 \pr /4 \le 2/17$ for all $\mathbf{z}$, and that each $T'_n$ is $2/17$-disjointness preserving to show that $T$ is $2/17$-disjointness preserving.

\smallskip

We are going to prove that if $S: C(X) \ra C(Y)$ is a weighted composition map, then $\vc T-S \vd \ge 1$. We suppose that this is not true, so there exist  
 $a \in C(Y)$ and a continuous map $h: c(a) \ra X$ such that $Sf  = a \cdot  f \circ h$ and $\vc T -S \vd <1$. We will see that this is not possible.

 \begin{claim}\label{tregueta}
The following hold:
\begin{enumerate}
\item $X_0 \cup \pl \bigcup_{n=1}^{\infty} - \mathbf{w}_n + X_0 \pr \subset c(a) $,
\item \label{nasda} Given $n \in \mathbb{N}$ and $\mathbf{y} \in X_0$,  $h \pl - \mathbf{w}_n + \mathbf{y} \pr =    - \mathbf{w}_{2n} + \mathbf{y} $. 
\item\label{entier-bebel} Given  $n \in \mathbb{N}$ and $E= A, B, C$. 
$h  \pl \mathbf{w}_n + E''_0 \pr = \mathbf{w}_n + E''_0$. 
\item $h(\mathbf{y}) = \mathbf{y}$ for every $\mathbf{y} \in X_0$.
\end{enumerate}
 \end{claim}
 
 \begin{proof}
 Suppose first that there exist $n \in \mathbb{N}$ and $\mathbf{y} \in  X_0$ with $- \mathbf{w}_n + \mathbf{y} \notin c(a)$. Consider $$f_0:= \xi_{- \mathbf{w}_{2n} + X_0} - \xi_{- \mathbf{w}_{2n-1} + X_0} \in C(X).$$ 
By definition, it is clear that $(Tf_0) (- \mathbf{w}_n + \mathbf{y}) =1$. Also $\vc f_0 \vd_{\infty} =1$ and $(Sf_0) (- \mathbf{w}_n + \mathbf{y}) =0$. This gives $\vc Tf_0 - Sf_0 \vd_{\infty} =1$, against our assumptions. 
 
 On the other hand, exactly the same contradiction is reached if we assume that $h \pl - \mathbf{w}_n + \mathbf{y} \pr \notin  \tl - \mathbf{w}_{2n-1} + \mathbf{y} , - \mathbf{w}_{2n} + \mathbf{y} \tr$ for some $n \in \mathbb{N}$ and $\mathbf{y} \in X_0$. Namely we take $g_0 \in C(X)$ with $\vc g_0 \vd_{\infty} =1$, and such that
$g_0 \pl \tl - \mathbf{w}_{2n-1} + \mathbf{y} , - \mathbf{w}_{2n} + \mathbf{y} \tr \pr \equiv 1$ and $g_0 \pl h \pl - \mathbf{w}_n + \mathbf{y}  \pr \pr =0$. It is clear that if we take $f_0$ as above, and  define $f_1 := f_0 g_0$, then $\vc f_1 \vd_{\infty} =1$
and $\vc Tf_1 - Sf_1 \vd_{\infty} =1$, which is impossible.
Of course, working now with $\xi_{- \mathbf{w}_{2n} + X_0}$, we deduce that 
$h \pl - \mathbf{w}_n + E''_0 \pr = - \mathbf{w}_{2n} + E''_0 $ for $E=A,B,C$, which implies, since
$h \pl  - \mathbf{w}_n + X_0 \pr $ is connected, that
$h \pl - \mathbf{w}_n + \mathbf{y} \pr = - \mathbf{w}_{2n} + \mathbf{y} $ for every $\mathbf{y} \in X_0$ and $n \in \mathbb{N}$. This proves (\ref{nasda}). The proof of (\ref{entier-bebel}) is similar.

 Suppose next that there is a point $\mathbf{y} \in X_0$ with $\mathbf{y} \notin c(a)$. Fix any $\delta >0$. Then there exists a neighborhoof $U$ of $\mathbf{y}$ such that $\va a (\mathbf{z}) \vb < \delta$ for every $\mathbf{z} \in U$. In particular we can select
 $\mathbf{z} \in U \cap \pl - \mathbf{w}_n + X_0 \pr $ for some $n \in \mathbb{N}$. Take $f_0$ as above, 
 which satisfies $(Tf_0) (\mathbf{z}) =1$.
Consequently, $\left| (Sf_0) (\mathbf{z}) \right| \le \va a(\mathbf{z}) \vb <\delta$ and 
$\va (Tf_0) (\mathbf{z}) - (Sf_0) (\mathbf{z}) \vb \ge 1- \delta$. We conclude again that $\vc T - S \vd =1$, against our assumptions.
 This shows that $X_0 \subset c(a)$. 
 
Finally, since $h$ is continuous, we deduce that $h(\mathbf{y}) = \mathbf{y}$ for every $\mathbf{y} \in X_0$.
\end{proof}

By Claim~\ref{tregueta}, we have that there exists $n_0 \in \mathbb{N}$ such that $\mathbf{w}_n + X_0 \subset c(a)$ for every $n \ge n_0$. 
Now, since $X_0$ is connected, we deduce in particular that for each $n \ge n_0$,  $h(\mathbf{w}_n + X_0) $ is 
contained in $ \mathbf{w}_{n} + X_{n}$.

If we now set $\mathbf{F}:= \tl (x, y, z) \in \mathbb{R}^3 : x^2 + y^2 <1/2 \tr$, we see that there is an open 
ball $B(D_0, r)$ of center  $D_0$ and radius $r \in (0, 1)$ such that $B(D_0, r) \subset c(a)$ and $h(B(D_0, r))  \subset \mathbf{F}$. Let
$n_1 \in \mathbb{N}$, $n_1 \ge n_0$, such that $B(D_0, r) \cap \pl \mathbf{w}_n + X_0 \pr \neq \emptyset$ for every $n \ge n_1$. 

We clearly have  that if we fix any 
$n \ge n_1$, then 
$$h \pl B(D_0, r) \cap \pl \mathbf{w}_{n} + X_0 \pr \pr \subset \mathbf{F} \cap \pl \mathbf{w}_{n} + X_{n} \pr.$$ 
On the other hand, $B(D_0, r) \cap \pl \mathbf{w}_{n} + X_0 \pr$ is connected, and so must be its image by $h$.
Since $\mathbf{F} \cap \pl \mathbf{w}_{n} + X_{n} \pr$ has three connected components, each containing a different point $\mathbf{w}_{n} + D_i^{n}$, $i=0,1,2$, then we have that $h \pl B(D_0, r) \cap \pl \mathbf{w}_{n} + X_0 \pr  \pr$ contains at most one point $\mathbf{w}_{n} + D_i^{n}$, $i=0,1,2$.

\begin{claim}\label{there-tarde}
The set of integers $n \ge n_1$ satisfying
$$\ca \tl i  : \mathbf{w}_{n} + D_i^{n} \in h \pl \mathbf{w}_n + X_0 \pr, i=0,1,2  \tr \ge 2$$
is finite.
\end{claim}

\begin{proof}
Suppose on the contrary that this set is infinite. 
By the comments above we deduce that there is an infinite subset $\mathbb{M}$ of $\mathbb{N}$ such that, if $n \in \mathbb{M}$, then there exists  $\mathbf{s}_n \in  \mathbf{w}_n + X_0  $, $\mathbf{s}_n \notin B(D_0 , r)$, 
and $h \pl \mathbf{s}_n \pr \in \tl \mathbf{w}_{n} + D_i^{n} : i=0,1,2 \tr$.  Since $Y$ is compact, there is an accumulation point
$\mathbf{s}$ of $\tl \mathbf{s}_n : n \in \mathbb{M} \tr$ in $X_0$, which necessarily satisfies $\vc \mathbf {s} \vd \ge r$. By continuity we must have $h(\mathbf{s}) =D_0$, and since $\mathbf{s} \in X_0$, then  we also have  $h (\mathbf{s}) = \mathbf{s}$, which is impossible.
\end{proof}

To finish, we use Claim~\ref{there-tarde}, and take an integer $n \ge n_1$ such that
there is at most one $i \in \{0,1,2\}$ with  $\mathbf{w}_{n} + D_i^{n} \in h \pl \mathbf{w}_n + X_0 \pr$. Suppose for instance that $i \neq 1, 2$ (the other cases are similar). By Claim~\ref{tregueta}(\ref{entier-bebel}), 
$h \pl \mathbf{w}_n + E''_0 \pr = \mathbf{w}_{n} +  E''_0 $, for $E=A,B,C$. 
so  the image by $h$ of the subset $\mathbf{w}_n + \pl \overline{D_0 A''_0} \cup \overline{D_0 C''_0} \pr$ 
is a connected subset of $ \mathbf{w}_{n} + \pl X_n  \setminus \tl D_1^{n}, D_2^{n} \tr \pr$ joining $\mathbf{w}_{n} +  A''_0 $ and $\mathbf{w}_{n} +  C''_0 $.
We easily see that this is impossible.
\end{ex}

\begin{ex}\label{gustavo}
{\em An example showing that the bound given in Theorem~\ref{rz} is sharp.}

Let $0 < \epsilon < 2/17$.  We claim that there exists a norm one
$\epsilon$-disjointness preserving operator $T'$  such that $\vc   T' - S' \vd   \ge \sqrt{17 \epsilon/2}$ for every  weighted composition map
 $S'$.

Let \[\gamma := \sqrt{\frac{17 \epsilon}{2}},\] and let $X$, $Y$, and $T$ 
be as in the previous example. We need a point not belonging to $X \cup Y$, for instance $(4, 0, 0)$. If we consider the sets $X' := X
\cup \{(4, 0, 0)\}$ and $Y' := Y \cup \{(4, 0, 0)\}$, then we 
define a linear map $T' : C(X') \longrightarrow C(Y')$ such that, for every $f\in C(X')$, $(T' f)
(4, 0, 0) := f(4, 0, 0)$  and, for all $\mathbf{y} \in Y$, $(T'f) (\mathbf{y}) := \gamma
(Tf_r) (\mathbf{y})$, where $f_r$ is the restriction of $f$ to $X$.

Since $T$ is a $ 2 /17$-disjointness preserving, then $T'$ is
$ 2\gamma^2 /17 $-disjointness preserving, which is to say
$\epsilon$-disjointness preserving.

Let $S' : C(X') \ra C(Y')$ be a weighted composition map with associated maps $a' \in C(Y')$ and $h' : Y' \ra X'$ (continuous on $c(a')$). It is easy to see that the set $A:=  c(a') \cap  h'^{-1} \pl X \pr$ is closed and open in $c(a')$, and that the restriction to $Y$ of $ a'  \xi_A $ (denoted by $a$)  belongs to $C(Y)$. Also, if we fix $x_0 \in X$ and define  $h :Y \ra X$  as $h(y) := h'(y)$ for every $y \in c(a)$, and $h(y) := x_0$ for $y \in Y \setminus c(a)$, then $h$ is continuous on $c(a)$. Next we  consider the weighted composition map $S: C(X) \ra C(Y)$ given as $Sf := a \cdot f \circ h$ for all $f \in C(X)$.  By Example~\ref{andalu} we have that, for every $\delta >0$, there exists $f_{\delta} \in C(X)$ with $\vc f_{\delta} \vd_{\infty} =1$ and $\vc \pl \gamma T - S \pr \pl f_{\delta} \pr \vd_{\infty} \ge \gamma- \delta$. It is now apparent that, if  $g_{\delta} \in C(X')$ is  an  extension of  $f_{\delta}$ such that $g_{\delta} (4, 0, 0) =0$, then
$ \vc \pl S' - T' \pr \pl g_{\delta} \pr \vd_{\infty}   \ge   \gamma- \delta$. Therefore $$\vc  S' - T' \vd \ge
\gamma=\sqrt{\frac{17 \epsilon}{2}}.$$
\end{ex}

\section{How far. The case when $X$ is infinite}\label{reallyfar}

 In this section we consider  the case when $X$ is infinite, and prove Theorems~\ref{ex3} and \ref{vr}. The proof
that Theorem~\ref{vr} is not valid for general $Y$ can obviously be seen  in Example~\ref{gustavo}, but also  in Example~\ref{destranyo-kajero}. The finite case is special, and we leave it for the next section.

\begin{proof}[Proof of Theorem~\ref{ex3}]
For $\delta >0$, let us choose a regular Borel  probability  measure $\mu$ on $X$
such  $\mu(\{x\})\le \delta / 2$ for every $x\in X$.

Next, fix $y_0,y_1$ in $Y$ and $x_0\in X$. After choosing two
disjoint neighborhoods, $U(y_0)$ and $U(y_1)$, of $y_0$ and $y_1$,
respectively, we define two continuous functions,
$\alpha:Y\longrightarrow [0,2 \sqrt{\epsilon}]$ and
$\beta:Y\longrightarrow [0,1]$, with the following properties:

\begin{itemize}
\item $\alpha(y_0)=2\sqrt{\epsilon}$
\item ${\rm supp}(\alpha)\subset U(y_0)$
\item $\beta(y_1)=1$
\item ${\rm supp}(\beta)\subset U(y_1)$
\end{itemize}

Next, for each $y\in Y$, we define two continuous linear
functionals on $C(X)$ as follows:

\[F_y(f)=\beta(y)\delta_{x_0}(f)\]
\[G_y(f)=\alpha(y)\int_{X}fd\mu\]

By using these functionals we can now introduce a linear map
$T:C(X)\longrightarrow C(Y)$ such that $(Tf)(y)=F_y(f)+G_y(f)$ for every $f \in C(X)$.

Let us first check that $\vc T\vd =1$. To this end, it is apparent
that $(T{\bf 1})(y_1)=F_{y_1}({\bf 1})+G_{y_1}({\bf 1})=1+0=1$.
Consequently, $\vc T\vd \ge 1$. On the other hand, it is easy to see that if $f\in C(X)$
satisfies $\vc f\vd_{\infty} =1$, then $\va (Tf)(y) \vb  \le 1$ for every $y \in Y$. 
Hence, $\vc T\vd =1$.

The next step consists of checking that $T$ is
$\epsilon$-disjointness preserving. Let $f,g\in C(X)$ with
$\vc f\vd_{\infty} =\vc g\vd_{\infty} =1$ and such that $c(f)\cap
c(g)=\emptyset$. 
It is easy to see that $(Tf)(y)(Tg)(y)=0$ whenever $y \notin U(y_0)$. On the other hand, if
$y\in U(y_0)$, then $\va (Tf)(y)(Tg)(y) \vb =\va G_y(f)\vb \va G_y(g)\vb$. It is clear that there exist two unimodular
scalars $a_1,a_2\in \mathbb{K}$ such that $a_1G_y(f)=\va G_y(f)\vb$ and
$a_2G_y(g)=\va G_y(g)\vb $. Since $\vc a_1f+a_2g\vd_{\infty}  =1$, then
\begin{eqnarray*}
\va G_y(f) \vb+ \va G_y(g) \vb &=& 
G_y(a_1f+a_2g)
\\ &=&
\alpha(y)\int_X(a_1f+a_2g)d\mu
\\ &\le&
\alpha(y)
\end{eqnarray*}

Consequently, $\va G_y(f)\vb \va G_y(g) \vb \le \alpha(y)^2/4$.
Indeed,
\[\va (Tf)(y)(Tg)(y) \vb=\va G_y (f) \vb \va G_y(g) \vb \le \frac{\alpha(y)^2}{4}\le \frac{(2\sqrt{\epsilon})^2}{4}=\epsilon\]

\medskip
Finally, we will see that  $\vc T-S\vd \ge
2\sqrt{\epsilon}(1-\delta)$ for every weighted composition map
$S:C(X)\longrightarrow C(Y)$.

Let $S \in \wc$, and let $h:c(S {\bf 1}) \ra X$ be its associated map. It is clear that, if $(S{\bf 1}) (y_0) =0$, then $\vc T-S \vd = \va (T-S) ({\bf 1}) (y_0) \vd = 2 \sqrt{\epsilon}$, so we may assume that $y_0 $ belongs to $c(S {\bf 1})$. By the regularity of the measure $\mu$, there exists
an open neighborhood $U$ of $h(y_0)$ such that $\mu(U)<\delta$.
Let us select $f\in C(X)$ satisfying
$0\le f\le 1$,  $f (h(y_0)) = 0$, and $f\equiv 1$ on $X\setminus U$.
Obviously  $(Sf)(y_0)=0$ and $\va (Tf)(y_0) \vb = \va G_{y_0}(f) \vb$. Hence
\begin{eqnarray*}
\vc T-S\vd  &\ge&
\va (Tf)(y_0) \vb 
\\ &\ge&
\alpha(y_0)\int_{X\setminus U}fd\mu
\\ &\ge&
2\sqrt{\epsilon}(1-\delta).
\end{eqnarray*}

This proves the first part. The second part is immediate because, since the measure can be taken atomless, then $\delta$ is as small as wanted.
\end{proof}

\begin{proof}[Proof of Theorem~\ref{vr}]
 We are assuming that there exists a discrete space $Z$ such that $Y= \beta Z$. Of course $Y$ may be finite (that is,  $Y=Z$), and this is necessarily the case  when we consider the second part of the theorem. Let $Z_0 := Z \cap Y_{2 \sqrt{\epsilon}} $, which is a nonempty closed and open subset of  $Z$, and $$Z_1 := \{z \in Z \setminus Z_0 : \exists x_z \in X \mbox{with} \va \lambda_{T_z} (\{x_z\}) \vb >0 \}.$$ Fix any $x_0 \in X$.
By Lemma~\ref{L1}, 
 we can define a map $h:Z \longrightarrow X$ such that
 $\va \lambda_{T_z} (\{h(z)\}) \vb \ge \sqrt{\vc T_z \vd^2 - 4 \epsilon}$ for every $z \in Z_0 $,  and such that $h(z) := x_z$ for $z \in Z_1$, and 
$h(z) := x_0$ for  $z \notin Z_0 \cup Z_1$. Also, since $Z$ is discrete, then $h$ is continuous, and consequently it can be extended to a continuous map from $Y$ to $X$ (when $Y \neq Z$). We will denote this extension also by $h$.

Define $\alpha : Z \longrightarrow \mathbb{K}$ as $\alpha (z) := \lambda_{T_z} (\{h(z)\})$ if $z \in Z_0 \cup Z_1$, and $\alpha (z) := 0$ otherwise, and extend it to a continuous function, also called $\alpha$, defined on $Y$. Then  consider  $S: C(X) \longrightarrow C(Y)$ defined as $(Sf) (y) := \alpha (y) f(h(y))$ for every $f \in C(X)$ and $y \in Y$.

\medskip

Let us check that
$\vc T- S\vd  \le 2\sqrt{ \epsilon}$.  Take  $f\in C(X)$
with $\vc f\vd_{\infty}\le 1$.
First, suppose that $z \in Z \setminus \pl Z_0 \cup Z_1 \pr$.  Then  $(Sf)(z) =0$, so
\[\va (Tf)(z)-( S f)(z) \vb = \va (Tf)(z) \vb \le 2 \sqrt{\epsilon}.\]
Now, if $z \in Z_1$, then $\vc T_z \vd \le 2 \sqrt{\epsilon}$ and, as in the proof of Lemma~\ref{llata}, 
\[\va (Tf)(z)-(Sf)(z) \vb  \le\vc T_z \vd  -  \va \lambda_{T_z} (\{h(z)\}) \vb < 2 \sqrt{\epsilon}.\]
On the other hand, if $z \in Z_0$, we know by Corollary~\ref{L3} that 
\[\va (Tf)(z)-(Sf)(z) \vb  \le\vc T_z \vd  -  \sqrt{\vc T_z\vd^2-4\epsilon}.\] By Lemma~\ref{pollo}, we have $\va (Tf)(z)-(Sf)(z) \vb  < 2 \sqrt{\epsilon}$ for every $z \in Z_0$. By continuity, we see that the same bound applies to every point in $Y$, and the first part is proved.

Finally, in the second case, that is, when $X$ does not admit an atomless regular Borel  probability  measure and $Y$ is finite, we have that $Y=Z$, and that $Z \setminus \pl Z_0 \cup Z_1 \pr$ consists of those points satisfying $\vc T_z \vd=0$. The conclusion is then easy.
\end{proof}

\section{The  case when $X$ is finite. How far}\label{cerilla}

In this section we prove  Theorems~\ref{recero} and \ref{cero}. The fact that Theorem~\ref{cero} does not hold for
arbitrary $Y$ (with more than one point) can be seen in next section (see Example~\ref{llegomalenayamigos2}).
\begin{proof}[Proof of Theorem~\ref{recero}]
We first prove the result when $n$ is odd. We follow the same ideas and notation as in the proof of Theorem~\ref{ex3}, with some differences. Namely,  we  directly take $\mu(\{x\}) = 1/n$ for every $x \in X$, and use a new function \[\alpha : Y \longrightarrow \ql 0, \min \tl \frac{2n \sqrt{\epsilon}}{\sqrt{n^2 -1}}, 1  \tr \qr \]
such that \[\alpha (y_0) = \min \tl \frac{2n \sqrt{\epsilon}}{\sqrt{n^2 -1}}, 1  \tr \]and
${\rm supp}(\alpha)\subset U(y_0)$. Notice that $\alpha (y_0 )= 2n \sqrt{\epsilon} / \sqrt{n^2 -1} $ if $\epsilon \le \omega_n$, and $\alpha (y_0 ) =1$ otherwise.

Clearly $\vc T \vd = 1$, and using the fact that \[\frac{(n-1)(n+1)}{4 n^2} = \max \tl \frac{l (n-l)}{ n^2} : 0 \le l \le n \tr, \]we easily see that  $T$ is $\epsilon$-disjointness preserving   both if $\epsilon \le \omega_n$ and if $\epsilon > \omega_n$. On the other hand, by the definition of the measure, reasoning as in the proof of Theorem~\ref{ex3}, we easily check that $\vc T-S \vd \ge \pl 1 - 1 /n \pr \alpha (y_0)$ for every  weighted composition $S$.

Finally,  we  follow the above pattern to prove the result when $n$ is even. In particular we also take $\mu(\{x\}) = 1/n$ for every $x \in X$, and use a function $\alpha : Y \longrightarrow [ 0, 2 \sqrt{\epsilon} ]$ with $\alpha (y_0) = 2 \sqrt{\epsilon}$ and ${\rm supp}(\alpha)\subset U(y_0)$. The rest of the proof  follows as above.
\end{proof}

\begin{proof}[Proof of Theorem~\ref{cero}]
Let $Z$ be a discrete space with $Y = \beta Z$. Since $X$ has $n$ points, say $X:= \{x_1, \ldots, x_n\}$, we have that, for  each $z \in Z$, $T_z$ is of the form $T_z := \sum_{i=1}^n a_i^z \delta_{x_i}$, for some $a_i^z \in \mathbb{K}$, $i = 1, \ldots, n$.
Consequently, for each $z \in Z$, we can choose a point $x_z \in X$ such that $\va \lambda_{T_z} (\{x_z\}) \vb \ge \va \lambda_{T_z} (\{x\}) \vb$ for every $x \in X$, which yields $\va \lambda_{T_z} (\{x_z\}) \vb \ge \vc T_z \vd /n$. This allows us to define a map $h: Z \longrightarrow X$ as $h(z) := x_z$ for every $z \in Z$. Since $h$ is continuous we can extend it to a continuous function defined on the whole $Y$, which we also call $h$.

Following a similar process as in the proof of Theorem~\ref{vr}, define $\alpha : Z \longrightarrow \mathbb{K}$ as $\alpha (z) := \lambda_{T_z} (\{h(z)\})$, and extend it to a continuous function defined on $Y$, also denoted by $\alpha$. 
Now,  define $S: C(X) \longrightarrow C(Y)$ as $(Sf) (y) := \alpha (y) f(h(y))$ for every $f \in C(X)$ and $y \in Y$.

\medskip 

Fix any $f \in C(X)$, $\vc f \vd_{\infty} \le 1$, and $z \in Z$. It is then easy to check that $\va (Tf) (z) - (Sf) (z )\vb \le (n-1) \vc T_z \vd /n $ . Consequently, if $\vc T_z \vd \le 2 \sqrt{\epsilon}$, we have \[\va (Tf) (z) - (Sf) (z )\vb \le  \frac{2(n-1)}{n}\sqrt{\epsilon} \le o'_X (\epsilon).\] 

\medskip
Let us now study the case when $\vc T_z \vd > 2 \sqrt{\epsilon}$. First, we know from  Corollary~\ref{L3} that $\va (Tf)(z)-(Sf)(z) \vb  \le \vc T_z \vd  -  \sqrt{\vc T_z\vd^2-4\epsilon}$. Next, we split the proof  into two cases.

\begin{itemize}
\item {\em Case 1. Suppose that $n$ is odd.} We see  that to finish the proof it is enough  to show  that \[\min \pl \vc T_z \vd  -  \sqrt{\vc T_z\vd^2-4\epsilon}, \frac{n-1}{n} \vc T_z \vd \pr \le   o'_X ( \epsilon)  \]
whenever $\vc T_z \vd > 2 \sqrt{\epsilon}$. To do this, we consider the functions 
$\gamma, \delta  :[2 \sqrt{\epsilon}, 1] \longrightarrow \mathbb{R}$ defined respectively as 
$\gamma (t) := t - \sqrt{t^2 -4 \epsilon}$, and $\delta (t) :=  (n-1) t/n $ for every 
$t \in [2 \sqrt{\epsilon}, 1]$. We have that $\gamma$ is decreasing (see Lemma~\ref{pollo}) and $\delta$ is increasing on the whole interval of definition. 

Now, if $\epsilon \le  \omega_n$,  then for $t_0 := \sqrt{\epsilon /\omega_n} \in \ql 2 \sqrt{\epsilon}, 1 \qr$, 
we have $\gamma (t_0)= \delta (t_0)$. This common value turns out to be 
$\delta (t_0) = 2\sqrt{ (n-1) \epsilon/(n+1)}$, that is, it is equal to $o'_X ( \epsilon)$, and we get that
 $\va (Tf) (z) - (Sf) (z )\vb \le o'_X (\epsilon)$ for every $z \in Z$.

On the other hand, if $\epsilon >  \omega_n$,  then 
$\delta (1) \le \gamma (1)$, so
$\delta (t) \le \gamma (t)$ for every $t \in [2 \sqrt{\epsilon}, 1]$, and $\va (Tf) (z) - (Sf) (z )\vb \le \delta (1) $ for every $z \in Z$. Since $\delta(1) = (n-1)/n = o'_X  ( \epsilon)$, we obtain the desired inequality also in this case.

\item {\em Case 2. Suppose that $n$ is even.} By Proposition~\ref{ksis}, we get that $\va \lambda_{T_z} (\{h(z)\}) \vb \ge \pl \vc T_z \vd + \sqrt{\vc T_z \vd^2 - 4 \epsilon} \pr \Big\slash  n$, so 
\[
\va (Tf) (z) - (Sf) (z )\vb \le \vc T_z \vd - \frac{\vc T_z \vd + \sqrt{\vc T_z \vd^2 - 4 \epsilon}}{n}.
\]

Consequently, to finish the proof in this case we just need to show that
\[\min \pl \vc T_z \vd  -  \sqrt{\vc T_z\vd^2-4\epsilon}, \vc T_z \vd - \frac{\vc T_z \vd + \sqrt{\vc T_z \vd^2 - 4 \epsilon}}{n} \pr \le    \frac{2 (n-1) \sqrt{\epsilon}}{n}. \]

Let $\eta  : [2 \sqrt{\epsilon}, 1] \longrightarrow \mathbb{R}$ be defined as \[\eta (t) := t - \frac{t +  \sqrt{t^2 -4 \epsilon}}{n}\] for every $t \in [2 \sqrt{\epsilon}, 1]$, and consider also the function $\gamma$ defined above.
Clearly, when $n=2$ we have $  \eta = \gamma /2$, and the above inequality follows from  Lemma~\ref{pollo}. So we assume that $n \neq 2$. 
We easily see that $\eta (t) \le \gamma (t) $ whenever $t \in \ql 2 \sqrt{\epsilon}, \sqrt{\epsilon /\omega_{n-1}} \qr$,
 and that  
$\eta$ is decreasing in $\ql 2 \sqrt{\epsilon}, \sqrt{\epsilon /\omega_{n-1}} \qr$ ($t \le 1$).
 We deduce that 
$$\min \pl \gamma (t) , \eta (t) \pr \le  \eta \pl 2 \sqrt{\epsilon} \pr = \frac{2 \pl n-1 \pr \sqrt{\epsilon}}{n}$$
whenever $2 \sqrt{\epsilon} \le t \le 1 $, as it was to be seen.
\end{itemize}

By denseness of $Z$ in $Y$, we conclude that $\vc T-S \vd  \le o'_X( \epsilon)$.
\end{proof}

\section{The  case when $X$ is finite. How close}\label{myrna}

	In this section we start proving Theorem~\ref{opodel}, and then we give an example showing that the bound
	given in it is in fact sharp. Of course this implies in particular that Theorem~\ref{cero} does not hold for $Y$
arbitrary, and consequently that the bounds for instability given in Theorem~\ref{recero} are not bounds for stability.

At the end of the section we provide  an example which shows that Theorem~\ref{opodel} is not valid in general for $X$ infinite, even in the simplest case, that is, when $X$ is  is a countable set with just one accumulation point. We see not only that  $2 \sqrt{\epsilon}$ is not a bound for stability, but that every bound for stability must be bigger than $\sqrt{8 \epsilon}$. This shows a dramatic passage from finite to infinite.

\begin{proof}[Proof of Theorem~\ref{opodel}]
We assume that $X = \{x_1 , \ldots, x_n \}$.
It is easy to see that $\vc T_y \vd = \sum_{i=1}^n \va \pl T \xi_{\{x_i\}} \pr \pl y \pr \vb$ for every $y \in Y$, and consequently the map from $Y$ to $\mathbb{K}$ given by $y \mapsto \vc T_y \vd$ is continuous.

\smallskip
For each set $C \subset X$, we consider $A_C := E_C \cap \pl \bigcap_{u \in C} E_C^u \pr$, where
\begin{eqnarray*}
E_C &:=& \tl y \in Y_{2 \sqrt{\epsilon}} : \va \lambda_{T_{y}} \vb \pl C \pr  \ge  \frac{\vc T_y \vd}{2} \tr
\\ &=& \tl y \in Y_{2 \sqrt{\epsilon}} : \sum_{x \in C} \va \pl  T  \xi_{\{x\}} \pr (y)  \vb \ge \frac{\sum_{i =1}^n \va \pl  T  \xi_{\tl x_i \tr } \pr (y)  \vb}{2} \tr , 
\end{eqnarray*}
and
\begin{eqnarray*}
E_C^u &:=& \tl y \in Y_{2 \sqrt{\epsilon}} : 
\va \lambda_{T_{y}} \vb \pl C\setminus \{ u\}   \pr   < \frac{\vc T_y \vd}{2} \tr
\\ &=& \tl y \in Y_{2 \sqrt{\epsilon}} : \sum_{x \in C \setminus \{u\}} \va \pl  T  \xi_{\{x\}} \pr (y)  \vb < \frac{\sum_{i =1}^n \va \pl  T  \xi_{\tl x_i \tr } \pr (y)  \vb}{2} \tr ,
\end{eqnarray*}

By Lemma~\ref{pol-immemoriam}, we know that $E_C$ coincides with the set of all $y \in Y_{2 \sqrt{\epsilon}}$ satisfying
$ \va \lambda_{T_{y}} \vb \pl C \pr  > \vc T_y \vd / 2$, that is, 
$$\sum_{x \in C} \va \pl  T  \xi_{\{x\}} \pr (y)  \vb >  \sum_{i =1}^n \va \pl  T  \xi_{\tl x_i \tr } \pr (y)  \vb /2 ,$$
and consequently is both open and closed as a 
subset of 
$Y_{2 \sqrt{\epsilon}}$.
In the same way, each $E_C^u$ is also open and closed in $Y_{2 \sqrt{\epsilon}}$, and so is $A_C$.

Notice   that again by Lemma~\ref{pol-immemoriam}, if $y \in A_C$, then 
$\va \lambda_{T_y} \vb \pl C \pr  \ge \pl \vc T_y \vd + 
\sqrt{\vc T_y \vd^2   - 4 \epsilon} \pr /2		$,  and 
$\va \lambda_{T_y} \vb \pl C \setminus \{u\} \pr  \le \pl \vc T_y \vd - 
\sqrt{\vc T_y \vd^2   - 4 \epsilon} \pr /2		$ for every $u \in C$. We conclude that 
$\va \lambda_{T_y} \vb \pl \tl u \tr \pr  \ge
\sqrt{\vc T_y \vd^2   - 4 \epsilon}$ for every $u \in C$.

On the other hand, it is clear that each element $y \in Y_{2 \sqrt{\epsilon}}$ belongs to some  $A_C$, so we
can make a finite partition of $Y_{2 \sqrt{\epsilon}}$ by open and closed sets $B_1, \ldots, B_m$, where each $B_i \subset A_C$ for some set $C$.     This implies that, for each $i =1, \ldots, m$, there exists a point $u_i \in X$ such that $\va \lambda_{T_y} (\{u_i\}) \vb \ge \sqrt{\vc T_y \vd^2 - 4 \epsilon}$
for every $y \in B_i$.  This allows us to define  a continuous map $h: Y_{2 \sqrt{\epsilon}} \ra X$ as $h (y) := u_i$ for every $y \in B_i$. Also take any  map $\mathbf{b} : Y \ra \mathbb{K}$ such that
$\mathbf{b} (y) = \lambda_{T_y} (\tl h(y) \tr )$ whenever $y \in Y_{2 \sqrt{\epsilon}}$, which is continuous
 on $Y_{2 \sqrt{\epsilon}}$.

\smallskip

We next follow a process similar to that seen in the proof of Theorem~\ref{rz}, with some necessary modifications. In particular we use  the map $\alpha \in C(Y)$ given  as 
$$ \alpha (y) := \sqrt{\frac{\vc T_y \vd - 2 \sqrt{\epsilon}}{\vc T_y \vd + 2 \sqrt{\epsilon}}}$$ for $y \in Y_{2 \sqrt{\epsilon}}$, and constantly as $0$ on $Y \setminus Y_{2 \sqrt{\epsilon}}$,
and  define 
 a weighted
composition map $S$ as \[(S f)(y):= \alpha(y) \mathbf{b} (y) f(h (y))\]
for all $f\in C(X)$ and $y\in Y$.

\smallskip

Now, for $y \in Y_{2 \sqrt{\epsilon}}$, put 
 $A_y:= \mathbf{b} (y) \delta_{h(y)}$.
It is easy to check that $\vc T_y\vd   = \vc T_y -A_y \vd   + \vc A_y \vd  $, and that, for $t \in [0,1]$ and
  $f \in C(X)$ with $\vc f \vd_{\infty} \le 1$, 
  \begin{eqnarray*}
\va (Tf)(y)- t  \mathbf{b} (y)  f(h(y)) \vb &\le& \va T_y f- A_y f \vb + \va A_y f - t A_y f \vb  
\\&\le&
\vc T_y -A_y\vd   + (1 - t) \vc A_y \vd     
\\ &=&
\vc T_y\vd  - t \vc  A_y\vd 
\\ &\le& \vc T_y \vd  - t \sqrt{\vc T_y\vd^2-4\epsilon},
\end{eqnarray*}

This allows us to use  the same arguments as in the proof of Theorem~\ref{rz}, and  show that 
$\vc T- S\vd  \le 2\sqrt{ \epsilon}$.  
\end{proof}

\begin{ex}\label{llegomalenayamigos2}
{\em An example showing that the bound given in  Theorem~\ref{opodel}  is sharp.}

Let $Y:= [-1,1]$ and $\epsilon  \in (0, 1/4)$. 
Take  two continuous and even functions
$\alpha : [-1,1] \ra \ql 2 \sqrt{\epsilon}, 1 \qr$ and $\beta : [-1,1] \ra \ql 1 , 1/ \sqrt{1-4 \epsilon}  \qr$, both increasing in $[0,1]$, such that  $\alpha (0) = 2 \sqrt{\epsilon}$, 
 $\alpha \pl 1 \pr = 1$, $\beta (0) =1$, and $\beta (1) = 1/ \sqrt{1 -4 \epsilon}$. Taking into account that $x \mapsto x / \sqrt{x^2 - 4 \epsilon}$ is decreasing for $x >2 \sqrt{\epsilon}$, we see that
 $\beta(t) \sqrt{\alpha^2 (t) - 4 \epsilon} \le \alpha (t)$ for every $t \in [-1, 1]$.

Now 
pick two points $ A, B \in X$ (recall  that we are assuming that
$X$ has at least two points), and
consider $T :C(X) \ra C(Y)$ such that, for every
 $f \in C(X)$, 
 \begin{eqnarray*}
(Tf) (t) &=&  \frac{\alpha (t)  + \sgn (t) \beta (t) \sqrt{ \alpha (t)^2 -4 \epsilon}}{2} \ f(A) \\ &+&  \frac{\alpha (t)  - \sgn (t) \beta (t) \sqrt{ \alpha (t)^2 -4 \epsilon}}{2} \ f(B) 
\end{eqnarray*}
 for every $t \in [-1,1]$, where $\sgn$ denotes the usual 
sign function.

It is clear that $T$ 
is $\epsilon$-disjointness preserving and has norm $1$. Also, since $\pl T \mathbf{1} \pr (\pm 1) = 1 $, it is easily seen that if a weighted
 composition map $S = a \cdot f \circ h$ is
at distance less than $2 \sqrt{\epsilon}$ from $T$, then    $1, -1 \in c(a)$.
On the other hand, if we suppose that $h(1) \neq A$, then we take $f_0 \in C(X)$ with $f_0 (A) =1 = \vc f_0 \vd_{\infty}$, and $f_0 (h(1)) =0 = f (B)$, and we see that $$\va (T - S) (f_0) (1)  \vb = 1 > 2 \sqrt{\epsilon}.$$
We deduce that, as $\vc T  - S \vd < 2 \sqrt{\epsilon}$, then $h(1) =A$, and in a similar way $h(-1) =B$.
Since $Y$ is connected and $h: c(a) \ra X$
 is continuous, we conclude that there is a point
 $t_0 \in Y$ such that
 $t_0 \notin c(a)$, that is, $(Sf) (t_0) =0$ for every $f \in C(X)$. Then it is easy to see that $\vc T - S \vd \ge \alpha (t_0) \ge  2 \sqrt{\epsilon} $.
 
 Notice that the above process is also valid if $X$ is infinite.
\end{ex}

\begin{ex}\label{destranyo-kajero}
{\em For $X= \mathbb{N} \cup \tl \infty \tr$ and any $\epsilon \in (0, 1/8)$, an $\epsilon$-disjointness preserving operator of norm $1$ whose distance to any weighted composition map is at least $\sqrt{8 \epsilon}$.}

 Given $r >0$, we denote by $C(r)$ the circle with center $0$ and radius $r$ in the complex plane. We take a strictly decreasing sequence $(r_n )$ in $\mathbb{R}$ converging to $0$ and the interval $[-r_1, 0]$, and define  $Y \subset \mathbb{C}$ as $$Y := [-r_1, 0] \cup \bigcup_{n=1}^{\infty} C(r_n).$$
 
 We also take $X := \mathbb{N} \cup \tl \infty \tr$.

\smallskip

Next let
\[\pi_0 :=\frac{1}{2} - \frac{ \sqrt{2}}{4},\]
and  consider a continuous map $\alpha : \bigcup_{n=1}^{\infty} C(r_n) \ra \ql 0, \pi_0 \qr$ such that $\alpha \pl -r_n \pr =0$
 and $\alpha \pl r_n \pr =\pi_0$ for every $n \in \mathbb{N}$.

Next, for each $f \in C(X)$ and $n \in \mathbb{N}$, we define, for $z \in C(r_n)$,
\[ (Tf) (z) :=
\pl \alpha (z) + \sqrt{2}/2 \pr f (2n) - \alpha (z) f (2n -1) .\]

On the other hand, if  $n \in \mathbb{N}$ 
and  $z \in (- r_n , - r_{n+1})$, then it is   of the form
\[z = - \pl t r_n  + (1-t) r_{n +1} \pr,\] where $t$ belongs to the open interval $
(0,1)$. In this case, we define
\begin{eqnarray*}
(Tf) (z) &:=&  t \pl Tf \pr \pl - r_n \pr + \pl 1-t \pr \pl Tf \pr \pl - r_{n+1} \pr  
\\&=& \frac{\sqrt{2}}{2}  \ql t f (2n)
+ (1- t) f (2n + 2 )  \qr. 
\end{eqnarray*}
Finally we put $$(Tf) (0) := \frac{\sqrt{2}}{2} f(\infty).$$

\smallskip
It is apparent that $T :C(X) \longrightarrow C(Y)$ is linear and
continuous, with $\vc T \vd =1$. Furthermore it is easy to see that if $f, g \in C(X)$ satisfy  $\vc f \vd_{\infty} =1 = \vc g \vd_{\infty}$ and $fg =0$, then $\va (Tf) (z) (Tg) (z) \vb \le 1/8$ for every $z \in Y$, that is, $T$ is $1/8$-disjointness preserving.

\smallskip
We will now check that we cannot find a weighted composition map "near"
$T$. Namely, if $S: C(X) \longrightarrow C(Y)$ denotes a weighted composition
map, then we claim that $\vc   S- T \vd   \ge 1$.

Let $D := c(S {\bf 1})$, and consider the continuous map $h :D
\longrightarrow X$ given by $S$. 
If $r_n \notin D$ for some $n \in \mathbb{N}$, then we take $f_n := \xi_{\{2n\}} - \xi_{\{2n-1\}}$.
 It is clear that $\vc   f_n \vd_{\infty}  =1$,
$(Tf_n) (r_n)=1$ and, as $r_n \notin D$, then $(Sf_n) (r_n) =0$. As a
consequence $\vc   S-T \vd   \ge 1$. It is also easy to see that we obtain the same conclusion if $h(r_n) \notin \{2n, 2n-1\}$. 

On the other hand, if we suppose that $0 \notin D$, 
then $(S {\bf 1}) (0) =0$. Therefore, given any $\delta
>0$, there exists a neighborhood $U$ of $0$ in $Y$ such that $\left|
(S{\bf 1}) (z) \right| < \delta$ for all $z \in U$. 
Choose now $r_n \in U$, and let $f_n$ be as above. It is apparent that either $({\bf 1} - f_n) (h(r_n)) = 0$ or
$({\bf 1} + f_n) (h(r_n)) = 0$,
which implies that $(S {\bf 1}) (r_n) - (Sf_n) (r_n) =0$ or $(S {\bf 1}) (r_n) + (Sf_n) (r_n) =0$.
Consequently, $\left| (Sf_n) (r_n) \right| =
\left| (S{\bf 1}) (r_n) \right| < \delta$ and, as in the previous cases, we easily deduce that  $\vc  S - T \vd   \ge 1-
\delta$. Therefore $\vc  S - T \vd   \ge 1$.

 Finally, we will see that we cannot have  $0 \in D$ and  $h (r_n) \in \{2n , 2n-1\}$ for every $n \in \mathbb{N}$. 
Otherwise, as $D$ is open, there
exists $s >0$ such that $B(0, s) \cap Y \subset D$. Also $h$ is continuous and $B(0, s) \cap Y $ is connected,
so $h \pl B(0, s) \cap Y \pr$ is constant. This is obviously impossible by our assumptions on $h (r_n)$.
 This contradiction shows that this case does not hold.
Hence,  we have $\vc  S - T \vd   \ge 1$. 

\smallskip

Let $0 < \epsilon < 1/8$.  We are going to construct a norm one
$\epsilon$-disjointness preserving map $T'$  such that, for all weighted composition map
 $S'$, $\vc   T' - S' \vd   \ge \sqrt{8 \epsilon}$.
Let \[\gamma := \sqrt{8 \epsilon},\] and let  $X' := X
\cup \{0\}$ and $Y' := Y \cup \{2r_1\} \subset \mathbb{C}$. 
Define a linear map $T' : C(X') \longrightarrow C(Y')$ as $(T' f)
(2 r_1) := f(0)$ and, for all $z \in Y$, $(T'f) (z) := \gamma
(Tf_r) (z)$, where $f_r$ is the restriction of $f$ to $X$.

Since $T$ is a $1/8$-disjointness preserving and $\epsilon = \gamma^2 /8$, then $T'$ is 
$\epsilon$-disjointness preserving. The conclusion follows as in Example~\ref{gustavo}.
\end{ex}

\section{Acknowledgements}
The authors would like  to thank Prof. Luis Alberto Fern\'andez for his help with the drawings.


\bigskip

\begin{thebibliography}{BNT}


\bibitem{ABN} J. Araujo, E. Beckenstein and L. Narici,
{\em Biseparating  maps and homeomorphic real-compactifications}. J. Math.
Anal. Appl. {\bf 192} (1995), 258--265.

\bibitem{Bour} D. G. Bourgin, {\em Approximately isometric and
multiplicative transformations
on continuous function rings}. Duke Math. J. {\bf 16} (1949), 385--397.


\bibitem{Dol} G. Dolinar, {\em Stability of disjointness preserving mappings}.
Proc. Amer. Math. Soc. {\bf 130} (2002), 129--138.

\bibitem{FH} J. J. Font and S. Hern\'andez, {\em On separating maps between
locally compact spaces}.
Arch. Math. (Basel) {\bf 63} (1994),
158--165.


\bibitem{HU} D. H. Hyers and S. M. Ulam, {\em On approximate isometries}.
Bull. Amer. Math. Soc. {\bf 51} (1945), 288--292.

\bibitem{HU2} D. H. Hyers and S. M. Ulam, {\em Approximate isometries of the
space of continuous functions}. Ann.
Math. {\bf 48} (1947), 285--289.


\bibitem{HIR} D. H. Hyers, G. Isac and Th. M. Rassias, {\em Stability of
functional equations in several variables}. Birkhauser, 1998.


\bibitem{Jar} K. Jarosz, {\em Perturbations of Banach algebras}. Lecture
Notes in
Mathematics 1120, Springer, Berlin, 1985.


\bibitem{Jz} K. Jarosz, {\em Automatic  continuity  of  separating  linear
isomorphisms}.
Canad. Math. Bull. {\bf 33}  (1990), 139--144.

\bibitem{JW} J.-S. Jeang and N.-C. Wong, {\em Weighted
composition operators of $C_0(X)$'s}.
J. Math. Anal. Appl. {\bf 201} (1996), 981--993.

\bibitem{Jo1} B. E. Johnson, {\em Approximately multiplicative functionals}.
J. London Math. Soc. (2) {\bf 34} (1986), 489--510.

\bibitem{Jo} B. E. Johnson, {\em Approximately multiplicative maps between
Banach algebras}.
J. London Math. Soc. (2) {\bf 37} (1988), 294--316.

\bibitem{Ru} W. Rudin, {\em Real and complex analysis}. 
Third edition. McGraw-Hill Book Co., New York, 1987. 

\bibitem{S} Z. Semadeni, {\em Banach Spaces of Continuous Functions}, Vol. I. Monograf. Mat. 55. PWN--Polish Sci. Publ., Warszawa, 1971.

\bibitem{Se} P. \v{S}emrl, {\em Nonlinear perturbations of homomorphisms on
$C(X)$}.
Quart. J. Math. Oxford Ser. (2) {\bf 50} (1999), 87--109.



\end{thebibliography}
\end{document}